\documentclass[12pt]{amsart}
\usepackage{amsmath,amssymb,amsfonts,amsthm,amsopn}
\usepackage{graphics}
\usepackage{xcolor}
\usepackage{pb-diagram}
\usepackage[title]{appendix}

\newcommand{\ft}{Fourier transform}

\newcommand{\tf}{time-frequency}

\newcommand{\tfs}{time-frequency shift}

\newtheorem{theorem}{Theorem}[section]
\newtheorem{lemma}[theorem]{Lemma}
\newtheorem{corollary}[theorem]{Corollary}
\newtheorem{proposition}[theorem]{Proposition}
\newtheorem{definition}[theorem]{Definition}

\newtheorem{example}[theorem]{Example}
\newtheorem{remark}[theorem]{Remark}

\newcommand{\beqa}{\begin{eqnarray*}}
\newcommand{\eeqa}{\end{eqnarray*}}

\newcommand{\field}[1]{\mathbb{#1}}
\newcommand{\bR}{\field{R}}        
\newcommand{\bC}{\field{C}}        
\def\la{\lambda}

\def\cF{\mathcal{F}}              
\def\cS{\mathcal{S}}
\def\cD{\mathcal{D}}

\def\cB{\mathcal{B}}

\def\cM{\mathcal{M}}

\def\cU{\mathcal{U}}
\def\cA{\mathcal{A}}

\def\cI{\mathcal{I}}
\def\cC{\mathcal{C}}
\def\cW{\mathcal{W}}

\def\hf{\hat{f}}

\def\rd{\bR^d}

\def\rdd{{\bR^{2d}}}

\def\R{\right)}

\def\<{\left<}
\def\>{\right>}

\def\mv1{M_v^1}

\def\mn{(m,n)}
\def\mn'{(m',n')}

\newcommand{\norm}[1]{\lVert#1\rVert}

\def\R{\mathbb{R}}
\def\Ren{\mathbb{R}^d}

\def\sch{\mathcal{S}}

\def\f{\varphi}

\def\Sn2{S_{2}(L^{2}(\Ren))}
\def\S1{S_{1}(L^{2}(\Ren))}
\def\sig00{\sigma_{0,0}}

\def\la{\langle}
\def\ra{\rangle}

\begin{document}
\begin{abstract} 
	We present a different symplectic point of view in the definition of weighted modulation spaces $M^{p,q}_m(\rd)$ and weighted Wiener amalgam spaces $W(\cF L^p_{m_1},L^q_{m_2})(\rd)$. All  the classical time-frequency representations, such as the short-time Fourier transform (STFT), the $\tau$-Wigner distributions and the  ambiguity function, can be written as metaplectic Wigner distributions $\mu(\cA)(f\otimes \bar{g})$, where $\mu(\cA)$ is the metaplectic operator and $\cA$ is the associated symplectic matrix. Namely, time-frequency representations can be represented as images of  metaplectic operators, which become  the real  protagonists of time-frequency analysis. In \cite{CR2022}, the authors suggest that any metaplectic Wigner distribution that satisfies the so-called \emph{shift-invertibility condition} can replace the STFT in the definition of modulation spaces. In this work, we prove that shift-invertibility alone is not sufficient, but it has to be complemented by an upper-triangularity condition for this characterization to hold, whereas a lower-triangularity property comes into play for Wiener amalgam spaces. The  shift-invertibility property is necessary: Rihacek and and conjugate Rihacek distributions are not shift-invertible and they fail  the characterization of the above spaces. We also exhibit examples of shift-invertible distributions without upper-tryangularity condition which do not define modulation spaces. Finally, we provide new families of time-frequency representations that characterize modulation spaces, with the purpose of replacing the time-frequency shifts with other atoms that allow to decompose signals differently, with possible new outcomes in applications.
\end{abstract}

\title[Symplectic Analysis of Time-Frequency Spaces]{Symplectic Analysis of Time-Frequency Spaces}

\author{Elena Cordero}
\address{Universit\`a di Torino, Dipartimento di Matematica, via Carlo Alberto 10, 10123 Torino, Italy}
\email{elena.cordero@unito.it}
\author{Gianluca Giacchi}
\address{Università di Bologna, Dipartimento di Matematica,  Piazza di Porta San Donato 5, 40126 Bologna, Italy; Institute of Systems Engineering, School of Engineering, HES-SO Valais-Wallis, Rue de l'Industrie 21, 1950 Sion, Switzerland; Lausanne University Hospital and University of Lausanne, Lausanne, Department of Diagnostic and Interventional Radiology, Rue du Bugnon 46, Lausanne 1011, Switzerland. The Sense Innovation and Research Center, Avenue de Provence 82
1007, Lausanne and Ch. de l’Agasse 5, 1950 Sion, Switzerland.}
\email{gianluca.giacchi2@unibo.it}

\thanks{}

\maketitle

\section{Introduction}



The modulation and Wiener amalgam spaces were introduced by H. Feichtinger in 1983 in his pioneering work \cite{F1} and  they started to become popular in the early 2000s in many different frameworks. In fact, they were successfully applied in the study of  pseudodifferential and Fourier integral operators, PDE's, quantum mechanics, signal processing.  Nowadays, the rich literature on these spaces witnesses their importance:  see, e.g., the very partial list of works \cite{BGGO2005,bhimanikasso2020,BHIMANI2021107995,wiener8, CGNRJMPA,PILIPOVIC2004194,RSTT2011,sugimototomita,Nenad2018,toft1}, as well as the textbooks \cite{KB2020,Elena-book,book,Gos11,wang}.

The modulation spaces $M^{p,q}_m(\rd)$ are classically defined in terms of the short-time Fourier transform (STFT), i.e.,
\begin{equation}\label{STFTp}
	V_gf(x,\xi)=\int_{\rd}f(t)\bar g(t-x)e^{-2\pi i\xi\cdot t}dt, \qquad f\in L^2(\rd), \ x,\xi\in\rd, 
\end{equation}
where $g\in L^2(\rd)\setminus\{0\}$ is a so-called window function and the definition is extended to $(f,g)\in\cS'(\rd)\times\cS(\rd)$ as in Section \ref{subsec:23}. Namely, for a tempered distribution $f\in\cS'(\rd)$,
\[
	f\in M^{p,q}_m(\rd) \quad \Longleftrightarrow \quad V_gf\in L^{p,q}_m(\rdd), 
\]
for an arbitrary fixed window $g$. The identity $V_gf(x,\xi)=\cF(f\cdot \bar g(\cdot-x))(\xi)$ justifies the choice of the STFT as the time-frequency representation used to define modulation spaces. In fact, it states that the STFT {can be used to measure} the local frequency content of signals in terms of weighted mixed norm spaces {(see Section \ref{subsec:MSs} below)}. \\

Apart of its interpretation,  there is no reason why the STFT shall serve as the leading time-frequency representation in the definition of modulation spaces. Actually, there are many reasons that make it unsuitable in many contexts, such as the theory of pseudodifferential operators and quantum mechanics, cf. \cite{Gos11,Shafi}. 

In  \cite{Gos11} M. De Gosson proved that the (cross-)Wigner distribution, defined for all $f,g\in L^2(\rd)$ as
\begin{equation}\label{Wignerp}
	W(f,g)(x,\xi)=\int_{\rd}f(x+\frac{t}{2})\bar g(x-\frac{t}{2}) e^{-2\pi i\xi\cdot t}dt, \qquad x,\xi\in\rd,
\end{equation}
can be used to define modulation spaces. Namely,
\[
	f\in M^{p,q}_m(\rd) \quad \Leftrightarrow \quad W(f,g)\in L^{p,q}_m(\rdd).
\]
Later, in  \cite{CR2021}, the above characterization was extended to  (cross-)$\tau$-Wigner distributions 
\[
W_\tau(f,g)(x,\xi)=\int_{\rd}f(x+\tau t)\bar g(x-(1-\tau)t) e^{-2\pi i\xi\cdot t}dt, \qquad x,\xi\in\rd,
\]
with $\tau\in\bR\setminus\{0,1\}$. The cases $\tau=0,1$ correspond to the so-called (cross-) Rihacek and conjugate Rihacek distributions, respectively. Their explicit expressions
\[
	W_0(f,g)(x,\xi)=f(x)\overline{\hat g(\xi)}e^{-2\pi i\xi\cdot x} \quad and \quad W_1(f,g)(x,\xi)=\hat f(\xi)\overline{g(x)}e^{2\pi i\xi\cdot x},
\]
$x,\xi\in\rd$, reveal that 
\[
	W_0(f,g)\in L^{p,q}(\rdd) \quad \Leftrightarrow \quad f\in L^p(\rd),
\]
as well as
\[
	W_1(f,g)\in L^{p,q}(\rdd) \quad \Leftrightarrow \quad f\in \cF L^p(\rd).
\]
The lowest common denominator of these time-frequency representations is that they can all be written as 
\begin{equation}\label{WApre}
W_\cA(f,g)=\mu(\cA)(f\otimes \bar g), 
\end{equation}
where $\mu(\cA)\in Mp(2d,\bR)$ is a so-called metaplectic operator and $\cA\in Sp(2d,\bR)$ is the unique associated symplectic matrix, we refer to Section \ref{subsec:26} for the precise definitions. In fact,
\[
	V_gf=\mu(A_{ST})(f\otimes\bar g) \quad \mbox{and} \quad W_\tau(f,g)=\mu(A_\tau)(f\otimes\bar g),
\]
where
\begin{equation}
	\label{AST}
	A_{ST}=\begin{pmatrix}
		I_{d\times d} & -I_{d\times d} & 0_{d\times d} & 0_{d\times d}\\
		0_{d\times d} & 0_{d\times d} & I_{d\times d} & I_{d\times d}\\
		0_{d\times d}  & 0_{d\times d}  & 0_{d\times d}  & -I_{d\times d} \\
		-I_{d\times d}  & 0_{d\times d} & 0_{d\times d} & 0_{d\times d} 
	\end{pmatrix}
\end{equation}
and
\begin{equation}\label{Atau}
A_{\tau}=\begin{pmatrix}
		(1-\tau)I_{d\times d} & \tau I_{d\times d} & 0_{d\times d} & 0_{d\times d}\\
		0_{d\times d} & 0_{d\times d} & \tau I_{d\times d} & -(1-\tau)I_{d\times d}\\
		0_{d\times d} & 0_{d\times d} & I_{d\times d} & I_{d\times d}\\
		-I_{d\times d} & I_{d\times d} & 0_{d\times d} & 0_{d\times d}
	\end{pmatrix}.
	\end{equation}
Using (\ref{WApre}) plenty of new time-frequency representations, that we call \textit{metaplectic Wigner distributions}, can be defined in terms of metaplectic operators, cf. \cite{CR2021, CR2022, CGR2022}. The question becomes for which $\mu(\cA)\in Mp(2d,\bR)$ or, equivalently, for which $\cA\in Sp(2d,\bR)$, the following property holds: for a moderate weight function $m$, $0<p,q\leq \infty$, and  a fixed non-zero $g\in\cS(\rd)$,
\begin{equation}\label{charMod}
	f\in M^{p,q}_m(\rd) \quad \Leftrightarrow \quad W_\cA(f,g)\in L^{p,q}_m(\rdd), \quad f\in \cS'(\rd).
\end{equation}
A partial answer is given in \cite{CR2022, CGR2022}, where the authors proved that for distributions of the form  $W_\cA(f,g)=\cF_2\mathfrak{T}_L(f\otimes\bar g)$ (see Example \ref{es22} below for their  definition), the characterization property holds. In \cite{CR2022}  it is conjectured that the characterization   (\ref{charMod}) should hold for $W_\cA$ satisfying the so-called \textit{shift-invertibility} property. Namely, for every $W_\cA$  the following equality holds
\begin{equation}\label{EA}
	|W_\cA(\pi(y,\eta)f,g)|=|W_\cA(f,g)(\cdot-E_\cA (y,\eta))|, \qquad (y,\eta)\in\rdd,
\end{equation}
for a suitable matrix $E_\cA\in\bR^{2d\times2d}$, where  $\pi(y,\eta)f(t)=e^{2\pi i\eta\cdot t}f(t-y)$ (see \eqref{eq:kh25} ahead).  $W_\cA$ is called  \emph{shift-invertible} if $E_\cA\in GL(2d,\bR)$ (that is, $E_{\cA}$ is invertible).\\

The main results of this work solve this conjecture for the Banach setting $1\leq p,q\leq\infty$. We summarize them is the following theorem.

\begin{theorem}\label{thmFp}
	Let $1\leq p,q\leq \infty$, $m$ be a $v$-moderate weight, $g$ a fixed non-zero window function in $\cS(\rd)$. Consider a metaplectic operator $\mu(\cA)\in Mp(2d,\bR)$ and let $\cA$ be the unique symplectic matrix associated to $\mu(\cA)$. The following statements hold.\\ 
	(i) If $m\circ E_\cA^{-1}\asymp m$ and $E_\cA\in GL(2d,\bR)$ in \eqref{EA} is upper triangular, then (\ref{charMod}) holds with equivalence of norms. \\
	(ii) If $(v\otimes v)\circ\cA^{-1}\asymp v\otimes v$, then the window class in $(i)$ can be enlarged to $M^1_v(\rd)\setminus\{0\}$.\\
	(iii) If $1\leq p=q\leq\infty$ and $E_\cA\in GL(2d,\bR)$, then (\ref{charMod}) holds with equivalence of norms.
\end{theorem}

The core of Theorem \ref{thmFp} is that shift-invertibility alone is not sufficient to characterize modulation spaces. This is not surprising: as it is observed in \cite{Fuhr}, $E_\cA$ has to be upper triangular, other than invertible, for the operator $f\mapsto f(E_\cA\cdot)$ to preserve the $L^{p,q}_m$ spaces. Nevertheless, we claim that the conditions on $E_\cA$ stated in Theorem \ref{thmFp} are fundamental to characterize modulation spaces. To support this thesis, we stress that the Rihacek distributions are examples of non shift-invertible metaplectic Wigner distributions for which Theorem \ref{thmFp} fails. 

Furthermore, we provide examples of shift-invertible metaplectic Wigner distributions
which characterize modulation spaces if and only if the matrix  $E_\cA$ is upper triangular, see Example \ref{upper-sharp} below (see also Remark \ref{remark3.8}).

A relevant contribution of this work consists of  constructing  explicit examples of metaplectic Wigner distributions that can be used to characterize modulation spaces, some of them extend the representations studied in \cite{Zhang21bis, Zhang21} (see also references therein). Our leading idea is to substitute the time-frequency shifts in (\ref{STFTp}) and (\ref{Wignerp}) with new \textit{time-frequency atoms}. Namely, we replace the chirp $e^{2\pi i\xi\cdot t}$ with the more general one $\Phi_C(\xi,t)=e^{i\pi (\xi,t)^T\cdot C(\xi,t)^{T}}$, with $C\in\bR^{2d\times 2d}$ symmetric matrix. The importance of these examples is that different atoms provide alternative ways to decompose signals into fundamental time-frequency functions, yielding to important applications in many branches of engineering, learning theory and signal analysis. In particular, discretization of time-frequency representations under this point of view could entail consequences in frame theory, phase retrival, and maybe in other aspects of signal processing, producing advances in these frameworks.

Besides the applications above, the metaplectic  approach to time-frequency representations  carry a high potential in many other situations were time-frequency representations play a crucial role, see e.g.,  \cite{Filippo2017,Cohen1,Filippocoorbite,ES89,LUEF2018288,MRomero2022}. 
Finally, observe that a first attempt to generalize the   $\tau$-Wigner distributions is contained in the work \cite{BCGT2022}, see also \cite{CT}.

\textbf{Outline.} Section \ref{sec:preliminaries} contains preliminaries and notation. The main results are exposed in Section \ref{sec:SIMS} whereas Section \ref{esempi} exhibits the most relevant examples. In the Appendix $A$ we extend some of the results in \cite{Fuhr} to general invertible matrices and to the quasi-Banach setting. In the Appendix $B$ we compute the matrices associated to {tensor products of metaplectic operators.} 

\section{Preliminaries}\label{sec:preliminaries}
\textbf{Notation.} We denote $t^2=t\cdot t$,  $t\in\rd$, and
$xy=x\cdot y$ (scalar product on $\Ren$).  The space   $\sch(\Ren)$ is the Schwartz class whereas $\sch'(\Ren)$  the space of temperate distributions.   The brackets  $\la f,g\ra$ denote the extension to $\sch' (\Ren)\times\sch (\Ren)$ of the inner product $\la f,g\ra=\int f(t){\overline {g(t)}}dt$ on $L^2(\Ren)$ (conjugate-linear in the second component). We write a point in the phase space (or \tf\ space) as
$z=(x,\xi)\in\rdd$, and  the corresponding phase-space shift (\tfs )
acts on a function or distribution  as
\begin{equation}
\label{eq:kh25}
\pi (z)f(t) = e^{2\pi i \xi\cdot t} f(t-x), \, \quad t\in\rd.
\end{equation}
$\cC_0^\infty(\rdd)$ denotes the space of smooth functions with compact support. The notation $f\lesssim g$ means that $f(x)\leq Cg(x)$ for all $x$. If $ g\lesssim f\lesssim g$ or, equivalently, $ f \lesssim g\lesssim f$, we write $f\asymp g$. For two measurable functions $f,g:\rd\to\bC$, we set $f\otimes g(x,y):=f(x)g(y)$. If $X,Y$ are vector spaces, $X\otimes Y$ is the unique completion of $\text{span}\{x\otimes y : x\in X, y\in Y\}$. If $X(\rd)=L^2(\rd)$ or $\cS(\rd)$, the set $\text{span}\{f\otimes g:f,g\in X(\rd)\}$ is dense in $X(\rdd)$. Thus, for all $f,g\in\cS'(\rd)$ the operator $f\otimes g\in\cS'(\rdd)$ characterized by its action on $\varphi\otimes\psi\in\cS(\rdd)$ by
\[
	\la f\otimes g,\varphi\otimes\psi\ra = \la f,\varphi\ra\la g,\psi\ra
\]
extends uniquely to a tempered distribution of $\cS'(\rdd)$. The subspace $\text{span}\{f\otimes g: f,g\in\cS'(\rd)\}$ is dense in $\cS'(\rdd)$.

\subsection{Weighted mixed norm spaces}
We denote by $v$ a continuous, positive, even, submultiplicative weight function on $\rdd$, i.e., 
$ v(z_1+z_2)\leq v(z_1)v(z_2)$, for all $ z_1,z_2\in\rdd$. Observe that since $v$ is even, positive and submultiplicative, it follows that $v(z)\geq1$ for all $z \in\rdd$. 
We say that $w\in \mathcal{M}_v(\rdd)$ if $w$ is a positive, continuous, even weight function  on $\rdd$  {\it
	$v$-moderate}:
$ w(z_1+z_2)\leq Cv(z_1)w(z_2)$  for all $z_1,z_2\in\rdd$. A fundamental example is the polynomial weight 
\begin{equation}\label{vs}
	v_s(z) =(1+|z|)^{s},\quad s\in\bR,\quad z\in\rdd.
\end{equation}
Two weights $m_1,m_2$ are equivalent if $m_1\asymp m_2$. For example, $v_s(z)\asymp (1+|z|^2)^{s/2}$.\\

If $m\in\cM_v(\rdd)$, $0<p,q\leq\infty$ and $f:\rdd\to\bC$ measurable, we set 
\[
	\norm{f}_{L^{p,q}_m}:=\left(\int_{\rd}\left(\int_{\rd}|f(x,y)|^pm(x,y)^p\right)^{q/p}dy\right)^{1/q}=\norm{y\mapsto\norm{f(\cdot,y)m(\cdot,y)}_{p}}_{q},
\]
with the obvious adjustments when $\min\{p,q\}=\infty$. The space of measurable functions $f$ having $\norm{f}_{L^{p,q}_m}<\infty$ is denoted by $L^{p,q}_m(\rdd)$. If $m\in\cM_v(\rdd)$ and $1\leq p,q\leq\infty$, then $L^{p,q}_{m}(\rdd)\ast L^1_v(\rdd)\hookrightarrow L^{p,q}_m(\rdd)$.

\subsection{Fourier transform}
In this work, the Fourier transform of $f\in \cS(\rd)$ is defined as
\[
	\hat f(\xi)=\int_{\rd} f(x)e^{-2\pi i\xi\cdot x}dx, \qquad \xi\in\rd.
\]
If $f\in\cS'(\rd)$, the Fourier transform of $f$ is defined by duality as the tempered distribution characterized by
\[
	\langle \hat f,\hat\varphi\rangle=\langle f,\varphi\rangle, \qquad \varphi\in\cS(\rd).
\]
We denote with $\cF f:=\hat f$ the Fourier transform operator. It is a surjective automorphism of $\cS(\rd)$ and $\cS'(\rd)$, as well as a surjective isometry of $L^2(\rd)$.

If $1\leq j\leq d$, the \textit{partial Fourier transform} with respect to the $j$th coordinate is defined  as
\begin{equation}\label{FJ}
	\cF_j f(t_1,\ldots, t_{j-1},\xi_j,t_{j+1}\ldots,t_d)=\int_{-\infty}^\infty f(t_1,\ldots,t_d)e^{-2\pi it_j\xi_j}dt_j,\quad f\in L^1(\rd).
\end{equation}
Analogously, the definition is transported on $\cS'(\rd)$ in terms of antilinear duality pairing: for all $f\in\cS'(\rd)$,
\[
	\langle \cF_j f,\varphi\rangle :=\langle f,\cF_j^{-1}\varphi\rangle, \qquad \varphi\in\cS(\rd).
\]
Observe that $\cF_j\cF_k=\cF_k\cF_j$ for all $1\leq j, k\leq d$. In particular, 
\begin{equation}\label{compFj}
	\cF=\cF_{\sigma(1)}\circ\ldots\circ\cF_{\sigma(d)}
	\end{equation}
	holds that for all permutation $\sigma\in Sym(\{1,\ldots,d\})$. Finally, for all $1\leq j\leq d$,
	\[
		\cF_j^2 f(x_1,\ldots,x_d)=f(x_1,\ldots,x_{j-1},-x_j,x_{j+1},\ldots,x_d).
	\]

\subsection{Time-frequency analysis tools}\label{subsec:23}
The \textit{short-time Fourier transform} of $f\in L^2(\rd)$ with respect to the window $g\in L^2(\rd)$ is the time-frequency representation defined in \eqref{STFTp}.
Its definition extends to $(f,g)\in\cS'(\rd)\times\cS(\rd)$ by antilinear duality as $V_gf(x,\xi)=\langle f,\pi(x,\xi)g\rangle$. Among all the equivalent definitions of $V_gf$, we recall that $$V_gf=\cF_2\mathfrak{T}_L(f\otimes\bar g),$$ where $\mathfrak{T}_LF(x,y)=F(y,y-x)$ and $\cF_2$ is the partial Fourier transform with respect to the second coordinate, cf. Example \ref{es22} below. This equality allows to extend the definition of $V_gf$ up to $\cS'(\rd)\times\cS'(\rd)$.

 We recall the fundamental identity of time-frequency analysis:
\begin{equation}\label{fundid}
	V_gf(x,\xi)=e^{-2\pi ix\cdot \xi}V_{\hat g}\hat f(J(x,\xi)), \qquad (x,\xi)\in\rdd,
\end{equation}
where the symplectic matrix $J$ is defined by 	\begin{equation}\label{defJ}
	J=\begin{pmatrix}
		0_{d\times d} & I_{d\times d}\\
		-I_{d\times d} & 0_{d\times d}
	\end{pmatrix}.
\end{equation}
	Here $I_{d\times d}\in\bR^{d\times d}$ is the identity matrix and $0_{d\times d}$ is the matrix of $\bR^{d\times d}$ having all zero entries.
The reproducing formula for the STFT reads as follows: for all $g,\gamma\in L^2(\rd)$ such that $\la g,\gamma\ra\neq0$,
\begin{equation}
	f(t)=\frac{1}{\la\gamma,g\ra}\int_{\rdd}V_gf(x,\xi)\pi(x,\xi)\gamma (t) dxd\xi,
\end{equation}
where the identity holds in $L^2(\rd)$ as a vector-valued integral in a weak sense (see, e.g., \cite[Subsection 1.2.4]{Elena-book}).

In high-dimensional complex features information processing $\tau$-Wigner distributions  ($\tau\in\bR$) play a crucial role \cite{ZJQ21}. They are defined as
\begin{equation}\label{tauWigner}
	W_\tau(f,g)(x,\xi)=\int_{\rd} f(x+\tau t)\overline{g(x-(1-\tau)t)}e^{-2\pi i\xi \cdot t}dt, \qquad  x,\xi\in\rd,
\end{equation}
for $f,g\in L^2(\rd)$.
The cases $\tau=0$ and $\tau=1$ are the so-called (cross-)\textit{Rihacek distribution} \cite{Guo-Zhao22}
\begin{equation}\label{RD}
W_0(f,g)(x,\xi)=f(x)\overline{\hat g(\xi)}e^{-2\pi i\xi\cdot x}, \quad x,\xi\in\rd,
\end{equation}
 and (cross-)\textit{conjugate Rihacek distribution}
 \begin{equation}\label{CRD}
 W_1(f,g)(x,\xi)=\hat f(\xi)\overline{g(x)}e^{2\pi i\xi\cdot x}, \quad x,\xi\in\rd.
 \end{equation}
 Observe that $W_\tau f(x,\xi)=\cF_2\mathfrak{T}_{L_\tau}(f\otimes\bar g)$, where for any $F$ on $\rdd$, $$\mathfrak{T}_{L_\tau}F(x,y)=F(x+\tau y,x-(1-\tau)y),\quad x,y\in\rd.$$

\subsection{Modulation  spaces \cite{KB2020,F1,Feichtinger_1981_Banach,book,Galperin2004,Kobayashi2006}} \label{subsec:MSs}
Fix $0<p,q\leq\infty$, $m\in\mathcal{M}_v(\rdd)$, and $g\in\cS(\rd)\setminus\{0\}$. The \textit{modulation space} $M^{p,q}_m(\rd)$ is classically defined as the space of tempered distributions $f\in\cS'(\rd)$ such that $$\norm{f}_{M^{p,q}_m}:=\Vert V_gf\Vert_{L^{p,q}_m}<\infty.$$ If $\min\{p,q\}\geq1$, the quantity $\norm{\cdot}_{M^{p,q}_m}$ defines a norm, otherwise a quasi-norm. Different windows give rise to equivalent (quasi-)norms. Modulation spaces are (quasi-)Banach spaces and the following continuous inclusions hold: \\
if $0<p_1\leq p_2\leq\infty$, $0<q_1\leq q_2\leq\infty$ and $m_1,m_2\in\mathcal{M}_{v}(\rdd)$ satisfy $m_2\lesssim m_1$: $$ \cS(\rd)\hookrightarrow M^{p_1,q_1}_{m_1}(\rd)\hookrightarrow M^{p_2,q_2}_{m_2}(\rd)\hookrightarrow\cS'(\rd).$$ In particular, $M^1_v(\rd)\hookrightarrow M^{p,q}_m(\rd)$ whenever $m\in\mathcal{M}_v(\rdd)$ and $\min\{p,q\}\geq1$. We will also use the inclusion $M^1_{m\otimes 1}(\rdd)\hookrightarrow L^1_m(\rdd)$. We denote with $\cM^{p,q}_m(\rd)$ the closure of $\cS(\rd)$ in $M^{p,q}_m(\rd)$, which coincides with the latter whenever $p,q<\infty$. Moreover, if $1\leq p,q<\infty$, $(M^{p,q}_m(\rd))'=M^{p',q'}_{1/m}(\rd)$, where $p'$ and $q'$ denote the Lebesgue conjugate exponents of $p$ and $q$ respectively. Finally, if $m_1\asymp m_2$, then $M^{p,q}_{m_1}(\rd)=M^{p,q}_{m_2}(\rd)$ for all $p,q$.\\

\subsection{Wiener amalgam spaces \cite{feichtinger-wiener-type,Feichtinger_1990_Generalized,Rauhut2007Winer}} For  $0<p,q\leq\infty$,  $m_1,m_2\in\cM_v(\rdd)$, the \textit{Wiener amalgam space} $W(\cF L^p_{m_1},L^q_{m_2})(\rd)$,   is  defined as the space of tempered distributions $f\in\cS'(\rd)$ such that for some (hence, all) window $g\in\cS(\rd)\setminus\{0\}$, 
\[
	\norm{f}_{W(\cF L^p_{m_1},L^q_{m_2})}:=\norm{ x\mapsto m_2(x)\norm{V_gf(x,\cdot)m_1}_p}_{q}<\infty.
\]
Using (\ref{fundid}), we have that $\norm{f}_{W(\cF L^p_{m_1},L^q_{m_2})}=\norm{\hat f}_{M^{p,q}_{m_1\otimes m_2}}$, so that $$\cF M^{p,q}_{m_1\otimes m_2}(\rd)=W(\cF L^p_{m_2},L^q_{m_1})(\rd).$$ 
Also, for $p=q$, 
\begin{equation}\label{WeM}
W(\cF L^p_{m_1},L^p_{m_2})(\rd)=M^p_{m_1\otimes m_2}(\rd).\end{equation}

\subsection{The symplectic group $Sp(d,\mathbb{R})$ and the metaplectic operators}\label{subsec:26}
	A matrix $\cA\in\bR^{2d\times 2d}$ is symplectic, we write $\cA\in Sp(d,\bR)$, if $\cA^TJ\cA=J$, where the matrix $J$ is defined in \eqref{defJ}. We represent $\cA$ as a block matrix 
	\begin{equation}\label{blocksA}
		\cA=\begin{pmatrix} A & B\\
		C & D\end{pmatrix}.
	\end{equation}
	If $\cA\in Sp(d,\bR)$, then $\det(\cA)=1$.
 The matrix $\cA\in Sp(d,\bR)$ with block decomposition \eqref{blocksA} is called \emph{free} if $\det B\not=0$.
  
	For $L\in GL(d,\bR)$ and $C\in\bR^{d\times d}$, $C$ symmetric, we define
	\begin{equation}\label{defDLVC}
		\cD_L:=\begin{pmatrix}
			L^{-1} & 0_{d\times d}\\
			0_{d\times d} & L^T
		\end{pmatrix} \qquad \text{and} \qquad V_C:=\begin{pmatrix}
			I_{d\times d} & 0\\ C & I_{d\times d}
		\end{pmatrix}.
	\end{equation}
	$J$ and the matrices in the form $V_C$ ($C$ symmetric) and $\cD_L$ ($L$ invertible) generate the group $Sp(d,\bR)$.\\ 

Let $\rho$ be the Schr\"odinger representation of the Heisenberg group, that is $$\rho(x,\xi,\tau)=e^{2\pi i\tau}e^{-\pi i\xi\cdot x}\pi(x,\xi),$$ for all $x,\xi\in\rd$, $\tau\in\bR$. For all $\cA\in Sp(d,\bR)$, $\rho_\cA(x,\xi,\tau):=\rho(\cA (x,\xi),\tau)$ defines another representation of the Heisenberg group that is equivalent to $\rho$, i.e. there exists a unitary operator $\mu(\cA):L^2(\rdd)\to L^2(\rdd)$ such that
\begin{equation}\label{muAdef}
	\mu(\cA)\rho(x,\xi,\tau)\mu(\cA)^{-1}=\rho(\cA(x,\xi),\tau), \qquad  x,\xi\in\rd, \ \tau\in\bR.
\end{equation}
This operator is not unique, but if $\mu'(\cA)$ is another unitary operator satisfying (\ref{muAdef}), then $\mu'(\cA)=c\mu(\cA)$, for some unitary constant $c\in\bC$, $|c|=1$. The set $\{\mu(\cA) : \cA\in Sp(d,\bR)\}$ is a group under composition and it admits a subgroup that contains exactly two operators for each $\cA\in Sp(d,\bR)$. This subgroup is called \textbf{metaplectic group}, denoted by $Mp(d,\bR)$. It is a realization of the two-fold cover of $Sp(d,\bR)$ and the projection \begin{equation}\label{piMp}
	\pi^{Mp}:Mp(d,\bR)\to Sp(d,\bR)
\end{equation} is a group homomorphism with kernel $\ker(\pi^{Mp})=\{-id_{{L^2}},id_{{L^2}}\}$.

\begin{proposition}{\cite[Proposition 4.27]{folland89}}\label{Folland427}
	The operator $\mu(\cA)\in Mp(2d,\bR)$ maps $\cS(\rd)$ isomorphically to $\cS(\rd)$ and it extends to an isomorphism on $\cS'(\rd)$.
\end{proposition}
For $C\in\R^{d\times d}$, define 
\begin{equation}\label{chirp}
	\Phi_C(t)=e^{\pi i t\cdot Ct},\quad t\in\rd.
\end{equation}
If we add the assumptions  $C$ symmetric and invertible, then we can compute explicitly its Fourier transform, that is  
\begin{equation}\label{ft-chirp}
\widehat{\Phi_C}=|\det(C)|\,\Phi_{-C^{-1}}.
\end{equation}

\begin{example}\label{es22} For particular choices of $\cA\in Sp(d,\bR)$, $\mu(\cA)$ is known. Let $J$, $\cD_L$ and $V_C$ be defined as in (\ref{defJ}) and (\ref{defDLVC}), respectively. Then, up to a sign,
	\begin{enumerate}
		\item[\it (i)] $\mu(J)f=\cF f$,
		\item[\it (ii)] $\mu(\cD_L)f=\mathfrak{T}_Lf=|\det(L)|^{1/2}\,f(L\cdot)$,
		\item[\it (iii)]  $\mu(V_C)f=\Phi_C f$,
		\item[\it (iv)] $\mu(V_C^T)f=\widehat{\Phi_{-C}}\ast f=\cF \Phi_{-C} \cF^{-1}$,
		\item[\it (v)] if $\cA_{FT2}\in Sp(2d,\bR)$ is the $4d\times4d$ matrix with block decomposition
		\[
		\cA_{FT2}:=\begin{pmatrix}
			I_{d\times d} & 0_{d\times d} & 0_{d\times d} & 0_{d\times d}\\
			0_{d\times d} & 0_{d\times d} & 0_{d\times d} & I_{d\times d} \\
			0_{d\times d} & 0_{d\times d} & I_{d\times d} & 0_{d\times d}\\
			0_{d\times d} & -I_{d\times d} & 0_{d\times d} & 0_{d\times d}
		\end{pmatrix},
		\]
		then,  $$\mu(\cA_{FT2})F(x,\xi)=\cF_2F(x,\xi), \quad x,\xi\in\rd,$$ the partial Fourier transform w.r.t. the second variable.
		\item[\it (vi)] Assume  $\cA\in Sp(d,\bR)$ with block decomposition (\ref{blocksA}). Then,\\
		if $\cA$ is free:
	\begin{equation}\label{free-int-rep}
		\mu(\cA)f(x)=(\det(B))^{-1/2}\Phi_{-DB^{-1}}(x)\int_{\rd}e^{2\pi i y\cdot B^{-1}x}\Phi_{-B^{-1}A}(y) f(y)dy.
	\end{equation}
		if $\det A\not=0$,
			\begin{equation}\label{Anonsing-int-rep}
		\mu(\cA)f(x)=(\det(A))^{-1/2}\Phi_{-CA^{-1}}(x)\int_{\rd}e^{2\pi i \xi\cdot A^{-1}x}\Phi_{-A^{-1}B}(\xi) \hf(\xi)d\xi.
			\end{equation}
	\end{enumerate}

\end{example}
Other important symplectic matrices are the so-called quasi-permutation matrices \cite{Dopico2009,Fuhr}.
\begin{definition}
	For $1\leq j\leq d$, the symplectic interchange matrix $\Pi_j\in Sp(d,\bR)$ is the  matrix obtained interchanging the columns $j$ and $j+d$ of the $2d$-by-$2d$ identity matrix and multiplying the $j$th column of the resulting matrix by $-1$.
\end{definition}
The corresponding metaplectic operators are the partial Fourier transforms, as we can see below.
\begin{example}\label{symplectic-interchange} Let $\cF_j$, $1\leq j\leq d$, be the partial \ft\, w.r.t. the $jth$ coordinate defined in \eqref{FJ}.	Then
	\begin{equation}\label{partialFT}
		\cF_j\rho(x,\xi,\tau)\cF_j^{-1} =\rho(\Pi_j(x,\xi),\tau),\quad x,\xi\in\rd,\,\tau\in\bR.
	\end{equation}
In fact, take any $f\in L^1(\rd)$ and compute $	\cF_j\rho(x,\xi,\tau)\cF_j^{-1}f$ as follows:
\begin{align*}\label{pij}
\cF_j&\rho(x,\xi,\tau)\cF_j^{-1}f(t_1,\dots,t_d)\notag\\
&=e^{2\pi i\tau} e^{-\pi i x\cdot\xi} e^{2\pi i\sum_{k\not=j} t_k\cdot\xi_k} \int_{\bR} e^{-2\pi t_j\zeta_j }e^{2\pi i \zeta_j\xi_j}\cF_j^{-1} f(t_1-x_1,\dots, \zeta_j-x_j,\dots,t_d-x_d)\,d\zeta_j\\
&= e^{2\pi i\tau} e^{-\pi i x\cdot\xi} e^{2\pi i\sum_{k\not=j} t_k\cdot\xi_k} \int_{\bR} e^{-2\pi i (u_j+x_j)(t_j-\xi_j)} \cF_j^{-1} f(t_1-x_1,\dots, u_j,\dots,t_d-x_d)\,du_j\notag\\
&= e^{2\pi i\tau} e^{-\pi i x\cdot\xi} e^{2\pi i\sum_{k\not=j} t_k\cdot\xi_k} e^{2\pi i t_j(-x_j)} e^{2\pi i x_j\xi_j}\int_{\bR} e^{-2\pi u_j(\xi_j-\zeta_j)}\notag \\
&\qquad\qquad\qquad\times\quad\cF_j^{-1} f(t_1-x_1,\dots, u_j,\dots,t_d-x_d)\,du_j\notag\\
&=\rho(\Pi_j(x,\xi),\tau)f(t_1,\dots,t_d).\notag
\end{align*}
Observe also that $\prod_j\Pi_j=J$, in line with (\ref{compFj}).
\end{example}

%

%

\section{Shift-invertibility and modulation spaces}\label{sec:SIMS}
In this section we present the features of metaplectic operators that guarantee the representations of modulation and Wiener amalgam spaces by metaplectic operators.
We first need to recall the definition of  metaplectic Wigner distributions and their main properties.
\subsection{Metaplectic Wigner distributions}
Metaplectic Wigner distributions were introduced and developed in \cite{CGR2022,CR2021,CR2022}, see also \cite{Zhang21bis,Zhang21}. The reader may recognize that Definition \ref{def1} below generalizes the STFT and the $\tau$-Wigner distributions (among others).

\begin{definition}\label{def1}
	Let $\mu(\cA)\in Mp(2d,\bR)$. The (cross-)\textbf{metaplectic Wigner distribution} with matrix $\cA$ (or $\cA$-Wigner distribution) is defined  as
	\[	
		W_\cA(f,g)(x,\xi)=\mu(\cA)(f\otimes \bar g)(x,\xi),\quad f,g\in L^2(\rd).
	\]
\end{definition}

The following results are direct consequences of Proposition \ref{Folland427}.

\begin{proposition}\label{prop32}
	Let $\mu(\cA)\in Mp(2d,\bR)$.\\
	(i) $W_\cA:L^2(\rd)\times L^2(\rd)\to L^2(\rdd)$ is continuous.\\
	(ii) $W_\cA:\cS(\rd)\times\cS(\rd)\to \cS(\rdd)$ is continuous.\\
	(iii) $W_\cA:\cS'(\rd)\times\cS'(\rd)\to\cS'(\rdd)$ is continuous.
\end{proposition}

We observe that \cite[Theorem 2.7]{CR2022} and \cite[Proposition 2.18]{CR2022} extend to  tempered distributions.

\begin{corollary}\label{cor3.3}
	Let $\mu(\cA)\in Mp(2d,\bR)$. Consider $g_1\in\cS'(\rd)$, $g_2\in\cS(\rd)$ with $\langle g_1,g_2\rangle\neq0$. Then, for all $f\in \cS'(\rd)$,
	\begin{equation}\label{invFor}
		f(x)=\frac{1}{\langle g_2,g_1\rangle}\int_{\rd}\mu(\cA)^{-1}W_\cA(f,g_1)(x,\xi)g_2(\xi)d\xi,
	\end{equation}
	with equality in $\cS'(\rd)$, the integral being meant in the weak sense.
\end{corollary}
\begin{proof}
	It is the same as in \cite[Theorem 2.7]{CR2022}. 
\end{proof}

\begin{corollary}\label{cor218}
Fix $g_3\in\cS(\rd)\setminus\{0\}$.	Under the assumptions of Corollary \ref{cor3.3}, for any $f\in\cS'(\rd)$,
	\[
		V_{g_3}f(w)=\frac{1}{\langle g_2,g_1\rangle}\langle W_\cA(f,g_2),W_\cA(\pi(w)g_3,g_2) \rangle, \qquad w\in\rdd,
	\]
	where $\la \cdot,\cdot\ra$ is the antilinear duality paring between $\cS'(\rdd)$ and $\cS(\rdd)$.
\end{corollary}
\begin{proof}
	It goes as in \cite[Proposition 2.18]{CR2022}, using Proposition \ref{prop32} $(ii)$ and $(iii)$.
\end{proof}

If $\mu(\cA)\in Mp(2d,\bR)$ with $\cA\in Sp(2d,\bR)$ having block decomposition
\begin{equation}\label{blockA}
	\cA=\begin{pmatrix}
		A_{11} & A_{12} & A_{13} & A_{14}\\
		A_{21} & A_{22} & A_{23} & A_{24}\\
		A_{31} & A_{32} & A_{33} & A_{34}\\
		A_{41} & A_{42} & A_{43} & A_{44}
	\end{pmatrix},
\end{equation}
then it was shown in \cite{CR2022} the equality
\begin{equation}\label{idSI}
	W_\cA(\pi(w)f,g)=\pi(E_\cA w,F_\cA w)W_\cA(f,g),\quad w\in\rdd,\quad\forall \,f,g\in L^2(\rd),
\end{equation}
 where the matrices $E_\cA$ and $F_\cA$ are given by
\begin{equation}\label{defEA}
	E_\cA=\begin{pmatrix}
		A_{11} & A_{13}\\
		A_{21} & A_{23}
	\end{pmatrix} \qquad and \qquad F_\cA=\begin{pmatrix}
		A_{31} & A_{33}\\
		A_{41} & A_{43}
	\end{pmatrix}.
\end{equation} 

\begin{definition}
	Under the notation above, we say that $W_\cA$ (or, by abuse, $\cA$) is \textbf{shift-invertible} if $E_\cA\in GL(2d,\bR)$.
\end{definition}

%
%

We need the following representation formula.

\begin{lemma}\label{intertFormula}
	Let $\mu(\cA)\in {Mp}(2d,\bR)$, $\gamma,g\in\cS(\rd)$ be such that $\langle \gamma,g\rangle\neq0$ and $f\in\cS'(\rd)$. Then,
	\begin{equation}
		W_\cA(f,g)=\frac{1}{\langle\gamma, g\rangle}\int_{\rdd}V_gf(w)W_\cA(\pi(w)\gamma,g) dw
	\end{equation}
	with equality in $\cS'(\rdd)$, the integral being intended in the weak sense.
\end{lemma}
\begin{proof}
	 Take any $\f\in\cS(\rdd)$ and use the definition of vector-valued integral in a weak sense which entails
	\begin{align*}
	\la W_\cA(f,g)(z), \f\ra &= \la f\otimes \bar g, \mu(\cA)^{-1}\f\ra\\
	&= \la \frac{1}{\langle\gamma,g\rangle}\int_{\rdd}V_gf(w)[(\pi(w)\gamma)\otimes \bar g ] dw, \mu(\cA)^{-1}\f\ra\\
	&=	 \frac{1}{\langle\gamma,g\rangle}\int_{\rdd}V_gf(w)\la [(\pi(w)\gamma)\otimes \bar g ], \mu(\cA)^{-1}\f\ra dw\\
	&=\frac{1}{\langle\gamma,g\rangle}\int_{\rdd}V_gf(w)\la \mu(\cA)(\pi(w)\gamma\otimes \bar g),\f\ra dw\\
		&=\frac{1}{\langle\gamma,g\rangle}\int_{\rdd}V_gf(w)\la W_\cA(\pi(w)\gamma,g)(z),\f\ra dw.
	\end{align*}
	Therefore,
	\begin{align*}
		 W_\cA(f,g)=\frac{1}{\langle\gamma,g\rangle}\int_{\rdd}V_gf(w) W_\cA(\pi(w)\gamma,g)dw
		\end{align*}
	with equality in $\cS'(\rdd)$. 
\end{proof}

\begin{theorem} \label{thmF}
	Let $W_\cA$ be shift-invertible with $E_\cA$ upper-triangular.  Fix a non-zero window function $g\in \cS(\rd)$. For $m\in\mathcal{M}_v(\rdd)$ with $m\asymp m\circ E_\cA^{-1}$, $1\leq p,q\leq\infty$,  
\begin{equation}\label{charmod}
		f\in M^{p,q}_m(\rd) \qquad \Leftrightarrow \qquad W_\cA(f,g)\in L^{p,q}_m(\rdd),
\end{equation}
	with equivalence of norms.
\end{theorem}
\begin{proof}
	$\Rightarrow$. Assume   $f\in M^{p,q}_m(\rd)$. For any $\gamma\in\cS(\rd)$ such that $\la \gamma,g\ra \not=0$, the inversion formula for the STFT (cf. Theorem 2.3.7 in \cite{Elena-book}) reads
	\[
		f=\frac{1}{\langle \gamma,g\rangle}\int_{\rdd}V_gf(w)\pi(w)\gamma dw.
	\]
Multiplying both sides of the above equality by $\bar g(z_2)$, for any $z=(z_1,z_2)\in\rdd$, we can write
	\begin{align*}
		f(z_1)\bar g(z_2)=(f\otimes \bar g)(z)&=\frac{1}{\langle\gamma,g\rangle}\int_{\rdd}V_gf(w)[(\pi(w)\gamma)\otimes \bar g ](z)dw.
	\end{align*}
	Applying $\mu(\cA)$ to $(f\otimes \bar g)$ we obtain $\mu(\cA)(f\otimes \bar g)=W_{\cA}(f,g)\in\cS'(\rdd)$. Using Lemma \ref{intertFormula}, we get:
	\[
		W_\cA(f,g)=\frac{1}{\langle\gamma, g\rangle}\int_{\rdd}V_gf(w)W_\cA(\pi(w)\gamma,g) dw,
	\]
	with equality holding in $\cS'(\rdd)$.

		Now, if  $f\in M^{p,q}_m(\rd)$, the integral on the right-hand side is absolutely convergent as we shall see presently. For any $z\in\rdd$,
	\begin{align}
		|W_\cA(f,g)(z)|&\leq\frac{1}{|\langle\gamma,g\rangle|}\int_{\rdd}|V_gf(w)||W_\cA(\gamma,g)(z-E_\cA w)|dw \nonumber \\
		&=\frac{|\det (E_\cA)|^{-1}}{|\langle\gamma,g\rangle|}\int_{\rdd}|V_gf(E_\cA^{-1} u)||W_\cA(\gamma,g)(z-u)|du \nonumber \\
		\label{2star}
		&=\frac{|\det (E_\cA)|^{-1}}{|\langle\gamma,g\rangle|}|V_gf \circ E_\cA^{-1}|\ast |W_\cA(\gamma,g)|(z).
	\end{align}
	Since $\gamma,g\in\cS(\rd)$, $W_\cA(\gamma,g)\in\cS(\rdd)\subset L^1_{v}(\rdd)$. Moreover, by Theorem \ref{thmA12} and Theorem \ref{thmA21} both applied with $S=E_\cA^{-1}$, we have that $V_gf\circ E_\cA^{-1} \in L^{p,q}_m(\rdd)$. Young's convolution inequality applied to \eqref{2star} entails
	\[
		\norm{W_\cA(f,g)}_{L^{p,q}_m}\lesssim \norm{V_gf}_{L^{p,q}_m}\norm{W_\cA(\gamma,g)}_{L^1_{v}}<\infty.
	\]
	$\Leftarrow.$ Assume that $W_\cA(f,g)\in L^{p,q}_m(\rdd)$. Using Corollary \ref{cor218} with $g_3=g_1=g$, $g_2=\gamma$,  for any $w\in \rdd$,
	\begin{align}
		|V_gf(w)|&\lesssim \frac{1}{|\langle \gamma,g\rangle|}|\langle W_\cA(f,g),W_\cA(\pi(w)g,\gamma)\rangle| \nonumber\\
		&\lesssim \int_{\rdd}|W_\cA(f,g)(u)||W_\cA(\pi(w)g,\gamma)(u)|du \nonumber\\
		&\lesssim \int_{\rdd} |W_\cA(f,g)(u)||W_\cA(g,\gamma)(u-E_\cA w)|du\nonumber\\
		&\lesssim \int_{\rdd} |W_\cA(f,g)(u)| |[W_\cA(g,\gamma)]^\ast(E_\cA w-u)|du\nonumber\\
		\label{3star}
		&=|W_\cA(f,g)|\ast |[W_\cA(g,\gamma)]^\ast|(E_\cA w).
	\end{align}
	Applying Theorem \ref{thmA12} and Theorem \ref{thmA21} with $S=E_\cA$, we obtain
	\begin{align*}
		\norm{f}_{M^{p,q}_m}&\asymp \norm{V_gf}_{L^{p,q}_m}\lesssim \norm{ |W_\cA(f,g)| \ast |[W_\cA(g,\gamma)]^\ast|}_{L^{p,q}_m}\\
		&\lesssim\norm{W_\cA(f,g)}_{L^{p,q}_m}\norm{W_\cA(g,\gamma)}_{L^1_{v}}<\infty,
	\end{align*}
	since we considered an even submultiplicative weight $v$.
\end{proof}

\begin{remark}\label{remark3.8}
	Theorem \ref{thmF} is sharp. Namely, if either $E_\cA$ is not shift-invertible or $E_\cA$ is not upper triangular, $W_\cA$ may not characterize modulation spaces. We provide two counterexamples. \\
	$(a)$ If $E_\cA$ is not shift-invertible, then $W_\cA$ may not characterize modulation spaces. Let $W_0$ be the (cross-)Rihacek distribution defined in \eqref{RD}. Obviously,
 for every $f\in L^p(\rd)$ and $g\in\cS(\rd)$, we obtain $\norm{W_0(f,g)}_{L^{p,q}}=\norm{f}_p\norm{\hat g}_q$. This means that  the $L^{p,q}$-norm of $W_0$ is not equivalent to the modulation norm in general. Observe that the corresponding matrix $E_{A_0}$ is not shift-invertible. In fact,
	\[
		E_{A_0}=\begin{pmatrix}
			I_{d\times d} & 0_{d\times d}\\
			0_{d\times d} & 0_{d\times d}
		\end{pmatrix}
	\]
	is not invertible. Similarly, the (cross-)conjugate-Rihacek distribution $W_1$ in \eqref{CRD} is not shift-invertible and does not characterize modulation spaces \cite[Remark 3.7]{CR2021}.\\
	$(b)$ If $E_\cA$ is not upper-triangular, then $W_\cA$ may not characterize modulation spaces. Let $C\in\bR^{2d\times 2d}\setminus\{0_{2d\times 2d}\}$ be any symmetric matrix. Then, up to a sign,
	\[
		V_g(\mu(V_C)f)=\mu(A_{ST})(\mu(V_C)f\otimes\bar g)=\mu(A_{ST}{V_{\tilde C}})(f\otimes \bar g),
	\]
	where
	\[
		{V_{\tilde C}}=\begin{pmatrix}
			I_{d\times d} & 0_{d\times d} & 0_{d\times d} & 0_{d\times d}\\
			0_{d\times d} & I_{d\times d} & 0_{d\times d} & 0_{d\times d}\\
			C & 0_{d\times d} & I_{d\times d} & 0_{d\times d}\\
			0_{d\times d} & 0_{d\times d} & 0_{d\times d} & I_{d\times d}
		\end{pmatrix},
	\]
	see formula (\ref{ce1}) in the Appendix $B$. Let $\cA:=A_{ST}{V_{\tilde C}}$ It is easy to verify that 
	\[
		E_{\cA}=\begin{pmatrix}
			I_{d\times d} & 0_{d\times d}\\
			C & I_{d\times d}
		\end{pmatrix},
	\]
	which is always invertible and lower-triangular. The metaplectic operator  $\mu(V_C)$ is unbounded on $M^{p,q}(\rd)$, $1\leq p,q\leq\infty$, $p\not=q$, cf. \cite[Proposition 7.1]{FIO1}. Namely, if $f\in M^{p,q}(\rd)$, $\mu(V_C)f\notin M^{p,q}(\rd)$ for $p\not=q$ and, consequently, $\mu(\cA)(f\otimes\bar g)\notin L^{p,q}(\rdd)$. Observe that a similar result with different methods is obtained in \cite[Theorem 3.3]{Fuhr}.
\end{remark}
As byproduct of the previous theorem we obtain new properties for shift-invertible representations $W_\cA$, see ahead.
\begin{corollary}
	Let $\mu(\cA)\in Mp(2d,\bR)$ with $W_\cA$ shift-invertible.   Then, for all $f,g\in L^2(\rd)$, we have $W_\cA(f,g)\in L^\infty(\rdd)$ and it is everywhere defined.
\end{corollary}
\begin{proof}
	If $f\in L^2(\rd)$ and $g,\gamma\in\cS(\rd)$, the inequality (\ref{2star}) holds pointwise (take $p=q=2$, $m=1$). Also, if $g\in L^2(\rd)$ the right hand-side of (\ref{2star}) is also well defined for all $z\in\rdd$, since $W_\cA$ maps $L^2(\rd)\times L^2(\rd)$ to $L^2(\rdd)$. By Young's inequality,
	\[
		\norm{W_\cA(f,g)}_{L^\infty(\rdd)}\lesssim\norm{V_gf\circ E_\cA^{-1}}_{L^2(\rdd)}\norm{W_\cA(\gamma,g)}_{L^2(\rdd)}=\norm{f}_2\norm{g}_2^2\norm{\gamma}_2<\infty.
	\]
	Hence, $W_\cA(f,g)\in L^\infty(\rdd)$ and $W_\cA(f,g)(z)$ is well defined for all $z\in\rdd$.
\end{proof}

If we limit to the case $p=q$, then $\mathfrak{T}_S:L^p(\rdd)\to L^p(\rdd)$ is bounded for all $S\in GL(2d,\bR)$, without any further assumption on its triangularity. In this case, arguing as above, but using Theorem \ref{thmA11}, we obtain the following result.

\begin{theorem}\label{thmPR}
	Let $W_\cA$ be shift-invertible and $m\in\cM_v(\rdd)$ with $m\asymp m\circ E_\cA^{-1}$. For $1\leq p\leq\infty$ and $g\in\cS(\rd)\setminus\{0\}$, we have
	\[
		\norm{f}_{M^p_m}\asymp\norm{W_\cA(f,g)}_{L^p_m}.
	\]
\end{theorem}

\begin{corollary}
	Under the assumptions of Theorem \ref{thmF}, assume that $(v\otimes v)\circ \cA^{-1}\asymp v\otimes v$, then the window class can be enlarged to $M^1_v(\rd)$. 
\end{corollary}
\begin{proof}
	By Theorem \ref{thmPR}, if $\gamma\in \cS(\rd)$ and $g\in M^1_v(\rd)$, $W_\cA(g,\gamma)\in L^1_v(\rdd)$, so that 
	\begin{equation}\label{star3}
		|W_\cA(f,g)(z)|\lesssim \frac{1}{|\langle\gamma,g\rangle|}|\det(E_\cA)|^{-1}|V_gf\circ E_\cA^{-1}|\ast|W_\cA(\gamma,g)|(z)
	\end{equation}
	is well defined by (\ref{2star}) provided that $W_\cA(\gamma,g)\in L^1_v(\rdd)$. 
	
	By \cite[Proposition 2.4]{CR2022}, 
	\[
		W_\cA(\gamma,g)=W_{\tilde\cA}(\bar g,\bar \gamma)=\mu(\cA)\mu(\cD_L)(\bar g\otimes \gamma),
	\]
	with $\tilde\cA=\cA\cD_L$. {Now, $\bar g\otimes \gamma\in M^1_v(\rd)\otimes \cS(\rd)\subset M^1_{v\otimes v}(\rdd)$ and $\mu(\cD_L)(\bar g\otimes \gamma)=\gamma\otimes \bar g$, so that $\mu(\cD_L):M^1_{v\otimes v}(\rdd)\to M^1_{v\otimes v}(\rdd)$}. Indeed,
	\[
		\norm{\mu(\cD_L)(\bar g\otimes \gamma)}_{M^1_{v\otimes v}}\asymp \norm{\gamma}_{M^1_{v}}\norm{g}_{M^1_{v}}.
	\]
	On the other hand, $\mu(\cA):M^1_{v\otimes v}(\rdd)\to M^1_{v\otimes v}(\rdd)$ by \cite[Theorem 3.2]{Fuhr} and \cite[Corollary 4.5]{Fuhr}. Moreover, $M^1_{v\otimes v}(\rdd)\hookrightarrow M^1_{v\otimes1}(\rdd)\hookrightarrow L^1_{v}(\rdd)$, since $$v(z)\leq v(z)v(w)$$ for all $z,w\in\rdd$. Hence,
	\begin{align*}
		\norm{W_\cA(\gamma,g)}_{L^1_v}&=\norm{W_{\tilde\cA}(\bar g,\bar \gamma)}_{L^1_v}\leq \norm{W_{\tilde \cA}(\bar g,\bar \gamma)}_{M^1_{v\otimes 1}}\leq \norm{W_{\tilde \cA}(\bar g,\bar \gamma)}_{M^1_{v\otimes v}}\\
		&\lesssim_\cA \norm{\mu(\cD_L)\bar g\otimes \gamma}_{M^1_{v\otimes v}}\asymp \norm{\bar g}_{M^1_{v}}\norm{\gamma}_{M^1_{v}}\asymp_{\cA,\gamma}\norm{g}_{M^1_v}<\infty.
	\end{align*}
	Going back to (\ref{star3}), we obtain
	\[
		\norm{W_\cA(f,g)}_{L^{p,q}_m}\lesssim \norm{V_gf}_{L^{p,q}_m}\norm{W_\cA(\gamma,g)}_{L^1_v}\asymp_{\cA,\gamma,g}\norm{f}_{M^{p,q}_m}.
	\]
	Whence, if $g\in M^1_v(\rd)$ and $f\in M^{p,q}_m(\rd)$, the $\cA$-Wigner $W_\cA(f,g)$ is in $L^{p,q}_m(\rdd)$, with $\norm{W_\cA(f,g)}_{L^{p,q}_m}\lesssim \norm{f}_{M^{p,q}_m}$.
	
	Vice versa, we have  shown that if $g\in M^1_v(\rd)$ and $f\in M^{p,q}_m(\rd)$,  the $\cA$-Wigner $W_\cA(f,g)$ is in $L^{p,q}_m(\rdd)$. By (\ref{3star}), for all $w\in\rdd$,
	\[
		|V_gf(w)|\lesssim|W_\cA(f,g)|\ast|[W_\cA(g,\gamma)]^\ast|(E_\cA w),
	\]
	and Young's inequality gives
	\[
		\norm{f}_{M^{p,q}_m}\lesssim_{\cA,g,\gamma}\norm{W_\cA(f,g)}_{L^{p,q}_m}.
	\]
	In conclusion, $\norm{f}_{M^{p,q}_m}\asymp\norm{W_\cA(f,g)}_{L^{p,q}_m}$, with $g\in M^1_v(\rd)$.
\end{proof}

Another consequence  of Theorem \ref{thmF} is the characterization of Wiener amalgam spaces $W(\cF L^p_{m_1},L^q_{m_2})(\rd)$.

\begin{corollary}\label{corWiener}
Let $\mu(\cA)\in Mp(2d,\bR)$ be such that $W_\cA$ is shift-invertible and $1\leq p, q\leq\infty$. Let $m_1,m_2\in\cM_v(\rdd)$ be such that $m_2\asymp \cI m_2$, being $\cI m_2(x)=m_2(-x)$, and $\cA=\pi^{Mp}(\mu(\cA))$ having block decomposition in \eqref{blockA}.
Fix $g\in\cS(\rd)\setminus\{0\}$ and define
	\begin{equation}\label{tEA0}
		\tilde E_\cA=-J E_\cA \begin{pmatrix} 0_{d\times d} & I_{d\times d} \\ I_{d\times d} & 0_{d\times d} \end{pmatrix}.
	\end{equation}
If $m_1\otimes m_2\asymp (m_1\otimes m_2)\circ \tilde E_{\cA}^{-1}$ and $E_\cA$ is lower triangular, then
\[
	\norm{f}_{W(\cF L^p_{m_1},L^q_{m_2})}\asymp\Big(\int_{\rd}\Big(\int_{\rd}|W_\cA(f,g)(x,\xi)|^pm_1(\xi)^pd\xi\Big)^{q/p}m_2(x)^qdx\Big)^{1/q},
	\]
	with the analogous for $\max\{p,q\}=\infty$.
\end{corollary}
\begin{proof}
	Assume that $\max\{p,q\}<\infty$. We use (\ref{fundid}). Let $f\in \cS'(\rd)$, $g\in \cS(\rd)\setminus\{0\}$. Then, 
	\begin{align*}
		\Big(\int_{\rd}\Big(\int_{\rd}&|W_\cA(f,g)(x,\xi)|^pm_1(\xi)^pd\xi\Big)^{q/p}m_2(x)^qdx\Big)^{1/q}\\
		&=\left(\int_{\rd}\left(\int_{\rd}|W_\cA(f,g)(J^{-1}(\xi,-x))|^pm_1(\xi)^pd\xi\right)^{q/p}m_2(x)^qdx\right)^{1/q}.
	\end{align*}
	Now, $|W_\cA(f,g)\circ J^{-1}|=|\mu(\cD_{J^{-1}})\mu(\cA)(f\otimes\bar g)|=|\mu(\cA_0)(f\otimes\bar g)|=|W_{\cA_0}(f,g)|$, where
	\begin{align*}
		\pi^{Mp}(\mu(\cA_0))=\cA_0:=\cD_{J^{-1}}\cA&=\begin{pmatrix}
		-A_{21} & -A_{22} & -A_{23} & -A_{24}\\
		A_{11} & A_{12} & A_{13} & A_{14}\\
		-A_{41} & -A_{42} & -A_{43} & -A_{44}\\
		A_{31} & A_{32} & A_{33} & A_{34}
		\end{pmatrix}.
	\end{align*}
	By \cite[Proposition 2.7]{CR2022}, $|W_{\cA_0}(f,g)|=|W_{\tilde\cA_0}(\hat f,\hat g)|$, where
	\[
		\tilde \cA_0 = \begin{pmatrix}
			-A_{23} & A_{24} & -A_{21} & A_{22}\\
			A_{13} & -A_{14} & A_{11} & -A_{12}\\
			-A_{43} & A_{44} & -A_{41} & A_{42}\\
			A_{33} & -A_{34} & A_{31} & -A_{32}
		\end{pmatrix}.
	\]
	Hence, using that $\cI m_2\asymp m_2$,
	\begin{align*}
		\Big(\int_{\rd}\Big(\int_{\rd}&|W_\cA(f,g)(x,\xi)|^pm_1(\xi)^pd\xi\Big)^{q/p}m_2(x)^qdx\Big)^{1/q}\\
		&\asymp\left(\int_{\rd}\left(\int_{\rd}|W_{\tilde \cA_0}(\hat f,\hat g)(\xi,-x)|^pm_1(\xi)^pd\xi\right)^{q/p}m_2(x)^qdx\right)^{1/q}\\
		&=\left(\int_{\rd}\left(\int_{\rd}|W_{\tilde \cA_0}(\hat f,\hat g)(\xi,x)|^pm_1(\xi)^pd\xi\right)^{q/p}\cI m_2(x)^qdx\right)^{1/q}\\
		&\asymp\norm{W_{\tilde\cA_0}(\hat f,\hat g)}_{L^{p,q}_{m_1\otimes m_2}}.
	\end{align*}
	Observe that $\tilde E_\cA = E_{\tilde \cA_0}$. Since $E_\cA$ is invertible and lower triangular the matrix $\tilde E_\cA$ in (\ref{tEA0}) is obviously invertible (and upper triangular). Hence, using the assumption $m_1\otimes m_2\asymp (m_1\otimes m_2)\circ E_{\tilde \cA_0}^{-1}$, we have
	\[
		\norm{W_{\tilde\cA_0}(\hat f,\hat g)}_{L^{p,q}_{m_1\otimes m_2}}\asymp\norm{\hat f}_{M^{p,q}_{m_1\otimes m_2}}=\norm{f}_{W(\cF L^p_{m_1},L^q_{m_2})}.
	\]
	The same argument also proves the case $\max\{p,q\}=\infty$, simply replacing the corresponding integrals with the essential supremums.
	
%
%
%
\end{proof}

\begin{remark}
	Because of (\ref{WeM}), Corollary \ref{corWiener} is significant only for $p\neq q$. For $p=q$ we refer to Theorem \ref{thmPR} with $m=m_1\otimes m_2$.
\end{remark}

\section{Examples}\label{esempi}
We exhibit a manifold of new metaplectic Wigner distributions which may find application in time-frequency analysis, signal processing, quantum mechanics and pseudodifferential theory.

\begin{example}\label{upper-sharp} This example generalizes the STFT by applying a metaplectic operator either on the window function $g$ or on the function $f$ as follows. First, consider the matrix 
	\[
	L=\begin{pmatrix}
		0_{d\times d} & I_{d\times d}\\
		-I_{d\times d} & I_{d\times d}
	\end{pmatrix}.
	\]
	which allows to rewrite the STFT $V_gf$ as composition of the metaplectic operators  $V_gf(x,\xi)=\cF_2\mathfrak{T}_L(f\otimes \bar g)(x,\xi)$. \par
	(i) We may act on the window $g$ by replacing $\bar{g}$ with $\mu(\cA')g$, $\mu(\cA')\in Mp(d,\bR)$. Namely, we consider the time-frequency representation
	\[
	\cU_gf(x,\xi)=\cF_2\mathfrak{T}_L(f\otimes\mu(\cA')\bar g)(x,\xi).
	\]
	Denoting 
	\begin{equation}\label{A'}
		\cA'=\begin{pmatrix}
			A' & B'\\ C' & D'
		\end{pmatrix},
	\end{equation}
	by (\ref{ce1}),
	\[
	f\otimes\mu(\cA')\bar g=\mu(\cA'')(f\otimes \bar g),
	\]
	with
	\[
	\cA''=\begin{pmatrix}
		I_{d\times d} &	0_{d\times d} & 0_{d\times d} & 0_{d\times d} \\
		0_{d\times d} & A' & 0_{d\times d} & B'\\
		0_{d\times d} & 0_{d\times d} & I_{d\times d} & 0_{d\times d} \\
		0_{d\times d} & C' & 0_{d\times d} & D'
	\end{pmatrix},
	\]
	so that $\cU_gf=W_\cA(f,g)$ with
	\[
	\cA=\begin{pmatrix}
		I_{d\times d} & A' & 0_{d\times d} & -B'\\
		0_{d\times d} & C' & I_{d\times d} & D'\\
		0_{d\times d} & -C' & 0_{d\times d} & -D'\\
		-I_{d\times d} & 0_{d\times d} & 0_{d\times d} & 0_{d\times d}
	\end{pmatrix},
	\]
	which is always shift-invertible with $E_\cA$ diagonal.
	This is not surprising, since $\mu(\cA')\bar{g}\in\cS(\rd)$ for $g\in\cS(\rd)$ and different windows in $\cS(\rd)$ yield equivalent norms.\par
	(ii) A more interesting example comes out by applying $\mu(\cA')$, with $\cA'\in Sp(d,\bR)$ having block decomposition in $\eqref{A'}$, to the function $f$. Namely, we consider
	$$\widetilde{\cU}_gf(x,\xi)=\cF_2\mathfrak{T}_L(\mu(\cA')f\otimes\bar g)(x,\xi).$$

	Then  $\widetilde{\cU}_gf=W_\cA(f,g)$ with
	\[
	\cA=\begin{pmatrix}
		A' & -I_{d\times d} &B' & 0_{d\times d} \\
		C' & 0_{d\times d} & D'  & I_{d\times d}\\
		0_{d\times d} &0_{d\times d}& 0_{d\times d} & -I_{d\times d}\\
		-	A' & 0_{d\times d} & -B'& 0_{d\times d}
	\end{pmatrix},
	\]
	and 
	\begin{equation*}
		E_\cA=\cA'=\begin{pmatrix}
			A' & B'\\ C' & D'
		\end{pmatrix},
	\end{equation*}
	in \eqref{A'}.  
	This $W_\cA$ characterizes modulation spaces if and only if the symplectic matrix $\cA'$ is upper triangular, since $\mu(\cA'): M^{p,q}(\rd)\to  M^{p,q}(\rd)$, $p\not=q$, if and only if  $\cA'$ is an upper block triangular matrix \cite{Fuhr}.
	
	Observe that these time-frequency representations find applications in signal processing, see    Zhang et al. \cite{Zhang21bis,Zhang21}.
\end{example}

\begin{example}\label{gSTFT}
	(i). For $z=(x,\xi)\in\rdd$, the time-frequency shift $\pi(z)$ can be written as follows: $\pi(z)g(t)=\Phi_{\tilde{I}}(\xi,t)T_xg(t)$, where $\tilde I\in\bR^{2d\times 2d}$ is the symmetric matrix
	\[	
		\tilde{I}=\begin{pmatrix}0_{d\times d} & I_{d\times d}\\ I_{d\times d} & 0_{d\times d}\end{pmatrix}.
	\]
	Thus, we can define a \textit{generalized STFT} replacing the time-frequency atoms $\pi(z)g$ with the more general atoms $\varsigma(x,\xi):=\Phi_C(\xi,\cdot)T_x$, $x,\xi\in\rd$, where $\Phi_C$ is the chirp function related to the symmetric matrix
	\[	
	C=\begin{pmatrix}C_{11} & C_{12}\\ C_{12}^T & C_{22}\end{pmatrix}
	\]
(hence $C_{11}^T=C_{11}$, $C_{22}^T=C_{22}$). Namely, we may define the  \textit{generalized STFT} $\mathcal{V}_{g,C}f$ as
	\[
		\mathcal{V}_{g,C}f(x,\xi)=|\det(C_{12})|^{1/2}e^{-i\pi C_{11}\xi\cdot \xi}\int_{\rd}f(t)\overline{g(t-x)}e^{-i\pi C_{22}t\cdot t}e^{-2\pi iC_{12}^T\xi\cdot t}dt=\langle f,\varsigma(x,\xi)g\rangle, 
	\]
	$f,g\in L^2(\rd)$. Observe that, if $C_{12}\in GL(d,\bR)$, then $\mathcal{V}_{g,C}f=W_{\cA}(f,g)$, with
	\[
		\cA=\begin{pmatrix}
			I_{d\times d} & -I_{d\times d} & 0_{d\times d} & 0_{d\times d}\\
			-C_{12}^{-T}C_{22} & 0_{d\times d} & C_{12}^{-T} & C_{12}^{-T}\\
			0_{d\times d} & 0_{d\times d} & 0_{d\times d} & -I_{d\times d}\\
			-C_{12}+C_{11}C_{12}^{-T}C_{22} & 0_{d\times d} & -C_{11}C_{12}^{-T} & -C_{11}C_{12}^{-T}
					\end{pmatrix},
	\]
	which is always shift-invertible, but unless $C_{22}\neq 0_{d\times d}$, $E_\cA$ is lower  triangular. \\
(ii) 
	For $\tau\in\bR$ consider the $\tau$-Wigner distribution defined in \eqref{tauWigner} and replace the Gabor atoms $\pi(x,\xi)$ with the more general chirp functions $\Phi_C$ as before to obtain
	\[
		\cW_{\tau,C}(f,g)(x,\xi)=|\det(C_{12})|^{1/2}e^{-i\pi C_{11}\xi\cdot\xi}\int_{\rd}f(x+\tau t)\overline{g(x-(1-\tau)t)}e^{-i\pi C_{22}t\cdot t}e^{-2\pi iC_{12}^T\xi\cdot t}dt,
	\]
	 $f,g\in L^2(\rd)$. Again, if $C_{12}\in GL(d,\bR)$, then $\cW_{\tau,C}(f,g)=W_\cA(f,g)$ with 
	 \[
	 	\cA=\begin{pmatrix}
			(1-\tau)I_{d\times d} & \tau I_{d\times d} & 0_{d\times d} & 0_{d\times d}\\
			-C_{12}^{-T}C_{22} & -C_{12}^{-T}C_{22} & \tau C_{12}^{-T} & -(1-\tau)C_{12}^{-T}\\
			0_{d\times d} & 0_{d\times d} & I_{d\times d} &  I_{d\times d} \\
			-C_{12}+C_{11}C_{12}^{-T}C_{22} & -C_{12}+C_{11}C_{12}^{-T}C_{22} & -\tau C_{11}C_{12}^{-T} &(1-\tau) C_{11}C_{12}^{-T}
		\end{pmatrix}.
	 \]
	 This matrix is shift-invertible if and only if $\tau\neq0,1$, and in this case $E_\cA$ is upper-triangular if and only if $C_{22}=0_{d\times d}$.
\end{example}

\begin{example}
	Every $\cA\in Sp(2d,\bR)$ can be written as $\Pi_{\mathcal{J}} V_Q\cD_LV_{-P}^T$, where $L\in GL(2d,\bR)$, the matrices $Q,P\in\bR^{2d\times 2d}$ are symmetric and, if $1\leq k\leq2d$, $1\leq j_1,\ldots,j_k\leq 2d$ and $\mathcal{J}=\{j_1,\ldots,j_k\}$, $\Pi_\mathcal{J}=\Pi_{j_1}\ldots\Pi_{j_k}$ is the matrix associated to the partial Fourier transform $\cF_{\mathcal{J}}:=\cF_{j_1}\ldots\cF_{j_k}$, cf. Example \ref{symplectic-interchange}. Set 
	\[
		Q=\begin{pmatrix}
			Q_{11} & Q_{12}\\
			Q_{12}^T & Q_{22}
		\end{pmatrix}, \ 
		P=\begin{pmatrix}
			P_{11} & P_{12}\\
			P_{12}^T & P_{22}
		\end{pmatrix}, \ L=\begin{pmatrix}
			L_{11} & L_{12}\\
			L_{21} & L_{22}
		\end{pmatrix} \ and \ L^{-1}=\begin{pmatrix}
			L'_{11} & L'_{12}\\
			L'_{21} & L'_{22}
		\end{pmatrix}.	
	\]
	A direct computation shows 
	\[
		V_Q\cD_LV_{-P}^T=
		\begin{pmatrix}
			L'_{11} & L'_{12} & -L_{11}'P_{11}^T-L_{12}'P_{12}^T & -L_{11}'P_{12} - L_{12}'P_{22}^T\\
			L'_{21} & L'_{22} & -L_{21}'P_{11}^T-L_{22}'P_{12}^T & -L_{21}'P_{12} - L_{22}'P_{22}^T\\
			M_{11} & M_{12} & N_{11} & N_{12}\\
			M_{21} & M_{22} & N_{21} & N_{22} 
		\end{pmatrix}.
	\]
	In what follows the explicit expressions of $M_{11},M_{12},M_{21},M_{22},N_{11},N_{12},N_{21},N_{22}\in\bR^{d\times d}$ are irrelevant. We consider the case $\mathcal{J}=\{d+1,\ldots,2d\}$, i.e., $\cF_{\mathcal{J}}=\cF_2$. The effect of left-multiplying $V_Q\cD_LV_{-P}^T$ by $\cA_{FT2}$ is to swap the second column blocks of $V_Q\cD_LV_{-P}^T$ with the fourth, up to change the sign of the latter. Hence, the matrix $E_\cA$ associated to
	\[
		W_\cA(f,g)=\cF_2(\Phi_Q\cdot(\cF^{-1}\Phi_P\ast (f\otimes\bar g)))
	\]
	is
	\[
		E_\cA=\begin{pmatrix}
			L_{11}' & -L_{11}'P_{11}^T-L_{12}'P_{12}^T\\
			L_{21}' & -L_{21}'P_{11}^T-L_{22}'P_{12}^T
		\end{pmatrix}.
	\]
	This matrix is upper triangular if and only if $L_{21}'=0_{d\times d}$ or, equivalently, if and only if $L$ is upper triangular. In this case, we also can compute explicitly $L^{-1}$ in terms of the blocks of $L$. Namely, 
	\[
		L=\begin{pmatrix}
			L_{11} & L_{12}\\
			0_{d\times d} & L_{22}
		\end{pmatrix}\in GL(2d,\bR) \Rightarrow L^{-1}=\begin{pmatrix}
			L_{11}^{-1} & -L_{11}^{-1}L_{12}L_{22}^{-1} \\
			0_{d\times d} & L_{22}^{-1}
		\end{pmatrix}.
	\]
	So, the corresponding $E_\cA$ is invertible if and only if $L_{22}^{-1}P_{12}^T\in GL(d,\bR)$, i.e., if and only if $P_{12}\in GL(d,\bR)$. In conclusion,  any metaplectic Wigner distribution of the form
	\[
		W_\cA(f,g)(x,\xi)=\cF_2(\Phi_Q\cdot(\cF^{-1}\Phi_P \ast (f\otimes\bar g)))(x,\xi)
	\]
	with $P,Q\in\bR^{2d\times2d}$ symmetric, $P_{12}\in GL(d,\bR)$, and $L\in GL(2d,\bR)$ upper triangular, can be used to define modulation spaces.
\end{example}

\begin{appendix} 
	\section{}\label{AppB} In Appendix \ref{AppB} we  generalize the results in \cite{Fuhr} to the quasi-Banach setting. Also, we observe that \cite[Corollary 4.2]{Fuhr} holds for general invertible matrices. For $S\in GL(2d,\bR)$, recall the definition of the metaplectic operator 
	$$\mathfrak{T}_S f(z)=|\det(S)|^{\frac{1}{2}}f(Sz),\quad z\in\rdd,$$
	defined in Example \ref{es22} $(ii)$.
	
	\begin{theorem}\label{thmA11}
		Let $S\in GL(2d,\bR)$ and $0<p\leq\infty$. The mapping $\mathfrak{T}_S:L^p(\rdd)\to L^p(\rdd)$ is everywhere defined and bounded with $\norm{\mathfrak{T}_S}_{L^p\to L^p}=|\det(S)|^{\frac{1}{2}-\frac{1}{p}}$. We use the convention $1/\infty=0$.
	\end{theorem}
	\begin{proof}
		Trivially, if $0<p<\infty$ and $f\in L^p(\rdd)$,
		\[
			\norm{\mathfrak{T}_Sf}_{L^p(\rdd)}=\norm{f\circ S}_{L^p(\rdd)}=|\det(S)|^{\frac{1}{2}-\frac{1}{p}}\norm{f}_{L^p(\rdd)}.
		\]
		Also, $\norm{\mathfrak{T}_Sf}_{L^\infty(\rdd)}=|\det(S)|^{1/2}\norm{f}_{L^\infty(\rdd)}$.
	\end{proof}
	
	\begin{theorem}\label{thmA12}
		Consider $A,D\in GL(d,\bR)$, $B\in \bR^{d\times d}$ and $0<p,q\leq\infty$. Define 
		\[
			S=\begin{pmatrix}
				A & B\\
				0_{d\times d} & D
			\end{pmatrix}.
		\]
		The mapping $\mathfrak{T}_S:L^{p,q}(\rdd)\to L^{p,q}(\rdd)$ is an isomorphism with bounded inverse $\mathfrak{T}_{S^{-1}}$.
	\end{theorem}
	\begin{proof}
		Let $f\in L^{p,q}(\rdd)$. Then, 
		\begin{align*}
			\norm{\mathfrak{T}_Sf}_{L^{p,q}(\rdd)}&=\norm{\xi\mapsto|\det(S)|^{1/2}\norm{f(A\cdot +B\xi,D\xi)}_{L^p(\rd)}}_{L^q(\rd)}\\
			&=\norm{\xi\mapsto|\det(S)|^{1/2}\norm{f(A\cdot,D\xi)}_{L^p(\rd)}}_{L^q(\rd)}\\
			&=\norm{\xi\mapsto|\det(S)|^{1/2}|\det(A)|^{-1/p}\norm{f(\cdot,D\xi)}_{L^p(\rd)}}_{L^q(\rd)}\\
			&=|\det(S)|^{1/2}|\det(A)|^{-1/p}|\det(D)|^{-1/q}\norm{\xi\mapsto\norm{f(\cdot,\xi)}_{L^p(\rd)}}_{L^q(\rd)}\\
			&=|\det(A)|^{\frac{1}{2}-\frac{1}{p}}|\det(D)|^{\frac{1}{2}-\frac{1}{q}}\norm{f}_{L^{p,q}(\rdd)},
		\end{align*}
		where $1/\infty=0$. Observe that $\mathfrak{T}_S^{-1}=\mathfrak{T}_{S^{-1}}$. It remains to prove that $\mathfrak{T}_{S^{-1}}$ is also bounded. Since $A$ (or, equivalently, $D$) is invertible, this follows by
		\[
			S^{-1}=\begin{pmatrix}
				A^{-1} & -A^{-1}BD^{-1}\\
				0_{d\times d} & D^{-1}
			\end{pmatrix}
		\]
		and by the first part of the statement.
	\end{proof}
	
	\begin{theorem}\label{thmA21}
	Let $m\in\mathcal{M}_v(\rdd)$, $S\in GL(2d,\bR)$ and $0<p,q\leq\infty$. Consider the operator $$(\mathfrak{T}_S)_m: f\in L^{p,q}_m(\rdd)\mapsto |\det(S)|^{1/2}f\circ S.$$ If $m\circ S\asymp m$, then
	$\mathfrak{T}_S:L^{p,q}(\rdd)\to L^{p,q}(\rdd)$ is bounded if and only if $(\mathfrak{T}_S)_m:L^{p,q}_m(\rdd)\to L^{p,q}_m(\rdd)$ is bounded.
	\end{theorem}
	\begin{proof}
		Observe that the condition $m\circ S\asymp m$ is equivalent to $m\circ S^{-1}\asymp m$. Assume that $\mathfrak{T}_S$ is bounded on $L^{p,q}(\rdd)$ and consider $f\in L^{p,q}_m(\rdd)$. Then, $fm\in L^{p,q}(\rdd)$ and
		\begin{align*}
			\norm{\mathfrak{T}_Sf}_{L^{p,q}_m(\rdd)}&=\norm{\mathfrak{T}_Sf \cdot m}_{L^{p,q}(\rdd)}=\norm{\mathfrak{T}_S(f\cdot (m\circ S^{-1}))}_{L^{p,q}(\rdd)}\\
			&\lesssim \norm{f\cdot (m\circ S^{-1})}_{L^{p,q}(\rdd)}=\left\|\left(f\frac{m\circ S^{-1}}{m}\right)m\right\|_{L^{p,q}(\rdd)}\\
			&=\left\|f\frac{m\circ S^{-1}}{m}\right\|_{L^{p,q}_m(\rdd)}\lesssim\norm{f}_{L^{p,q}_m(\rdd)}.
		\end{align*}
		For the converse, assume that $(\mathfrak{T}_S)_m:L^{p,q}_m(\rdd)\to L^{p,q}_m(\rdd)$ is bounded and take $f\in L^{p,q}(\rdd)$. Then, $f/m\in L^{p,q}_m(\rdd)$ and
		\begin{align*}
			\norm{\mathfrak{T}_Sf}_{L^{p,q}(\rdd)}&\lesssim \left\|\mathfrak{T}_Sf\left(\frac{m}{m\circ S}\right)\right\|_{L^{p,q}(\rdd)}\asymp\left\|\mathfrak{T}_S\left(\frac{f}{m}\right)\right\|_{L^{p,q}(\rdd)}\\
			&=\left\|\mathfrak{T}_S\left(\frac{f}{m}\right)\right\|_{L^{p,q}_m(\rdd)}\lesssim \norm{f/m}_{L^{p,q}_m(\rdd)}=\norm{f}_{L^{p,q}(\rdd)}.
		\end{align*}
	\end{proof}
	
	\section{}
	In this Appendix we study tensor products of metaplectic operators and refer to \cite{Kadison} for the theory of tensor products of HIlbert spaces. We are interested in proving the following result.
	\begin{theorem}\label{thmOtimes}
		Let $\mu(\cA),\mu(\cB)\in Mp(d,\bR)$ with $\cA=\pi^{Mp}(\mu(\cA))$ and $\cB=\pi^{Mp}(\mu(\cB))$ having block decompositions
		\[
		\cA=\begin{pmatrix}
			A & B\\ C & D
		\end{pmatrix},\  and \  \cB=\begin{pmatrix}
		E & F \\
		G & H
		\end{pmatrix}.
	\]
	Then, the bilinear operator $S:L^2(\rd)\times L^2(\rd)\to L^2(\rdd)$ defined for all $f,g\in L^2(\rd)$ as
	\[
		S(f,g)=\mu(\cA)f\otimes\mu(\cB)g
	\]
	extends uniquely to a metaplectic operator $\mu(\cC)\in Mp(2d,\bR)$ with $\cC=\pi^{Mp}(\mu(\cC))$ having block decomposition
	\[	\cC=\begin{pmatrix}
			A & 0_{d\times d} & B & 0_{d\times d}\\
			0_{d\times d} & E & 0_{d\times d} & F\\
			C & 0_{d\times d} & D & 0_{d\times d}\\
			0_{d\times d} & G & 0_{d\times d} & H
		\end{pmatrix}.
	\]
	\end{theorem}
	\begin{proof}
		By \cite[Proposition 2.6.6]{Kadison}, there exists a unique linear mapping $T:L^2(\rd)\otimes L^2(\rd)=L^2(\rdd)\to L^2(\rdd)$ satisfying
		\[
			T(f\otimes g)=\mu(\cA)f\otimes\mu(\cB)g, \qquad f,g\in L^2(\rd).
		\]
		By \cite[Proposition 2.6.12]{Kadison}, this extension is also bounded. Moreover, $T$ is invertible because $\mu(\cA)$ and $\mu(\cB)$ are. In particular, $T$ is surjective. To prove that $T$ is a metaplectic operator, it remains to check that $T$ preserves the $L^2$ inner product and that 
		\begin{equation}\label{meta}
			T\rho( z,\tau)=\rho(\cC z,\tau)T, \qquad z\in\bR^{4d}, \ \tau\in\bR.
		\end{equation}
		For all $f,g,\varphi,\psi\in L^2(\rd)$,
		\begin{align*}
			\la T(f\otimes g),T(\varphi\otimes\psi)\ra&=\la\mu(\cA)f,\mu(\cA)\varphi\ra\la\mu(\cB)g,\mu(\cB)\psi\ra=\la f,\varphi\ra\la g,\psi \ra\\
			&=\la f\otimes g,\varphi\otimes\psi\ra.
		\end{align*}
		If $\Phi\in L^2(\rdd)$, $\Phi=\sum_{j=1}^\infty c_j\varphi_j\otimes\psi_j$, with the sequence $(c_j)_j\subseteq\bC$ vanishing definitely,
		\begin{align*}
			\la T(f\otimes g),T\Phi\ra&=\sum_j\bar c_j \la T(f\otimes g),T(\varphi_j\otimes \psi_j)\ra=\sum_j \bar c_j \la f\otimes g,\varphi_j\otimes \psi_j\ra\\
			&=\la f\otimes g,\Phi\ra.
		\end{align*}
		Now, consider $\Phi\in L^2(\rdd)$ and  $(\Phi_j)_j\subseteq \text{span}\{\varphi\otimes\psi : \varphi,\psi\in L^2(\rd)\}$ satisfy $\lim_{j\to+\infty}\norm{\Phi-\Phi_j}_{L^2(\rdd)}=0$. Then, by the continuity of $T$ and of the inner product,
		\begin{align*}
			\la T(f\otimes g),T\Phi\ra&=\la T( f\otimes g), T(\lim_{j\to+\infty}\Phi_j)\ra=\lim_{j\to+\infty}\la T( f\otimes g),T\Phi_j\ra\\
			&=\lim_{j\to+\infty}\la f\otimes g,\Phi_j\ra=\la f\otimes g,\Phi\ra.
		\end{align*}
		So, we proved that
		\begin{equation}\label{unit}
			\la T F, T\Phi\ra=\la F,\Phi\ra
		\end{equation}
		holds for all $F=f\otimes g$, $f,g\in L^2(\rd)$ and all $\Phi\in L^2(\rdd)$. The same argument applied to the first component of the inner product shows that (\ref{unit}) holds for all $F\in L^2(\rdd)$ as well. So, $T$ is surjective and preserves the inner product, hence it is unitary. It remains to prove (\ref{meta}), which states that $T$ is a metaplectic operator with $\pi^{Mp}(T)=\cC$.
		
		For, consider $f,g\in L^2(\rd)$,  $\tau\in\bR$, $z=(x_1,x_2,\xi_1,\xi_2)\in\bR^{4d}$ and $z_j=(x_j,\xi_j)\in\rdd$ ($j=1,2$). 
		First, observe that 
\begin{equation}\label{tens1}
	\rho(z,\tau)(f\otimes g)=e^{-2\pi i\tau}\rho(z_1,\tau)f\otimes\rho(z_2,\tau)g
\end{equation}
and
\[
	\pi(\cC z)(f\otimes g)=\pi(\cA z_1)f\otimes \pi(\cB z_2)g,
\]
so that
\begin{align*}
	\rho(\cC z,\tau)T(f\otimes g)&=e^{2\pi i\tau}e^{-i\pi(Ax_1+B\xi_1)(Cx_1+D\xi_1)}e^{-i\pi(Ex_2+F\xi_2)(Gx_2+H\xi_2)}\\
	&\qquad\qquad\times \pi(\cC z)(\mu(\cA)f\otimes \mu(\cB)g)\\
	&=e^{2\pi i\tau}e^{-i\pi(Ax_1+B\xi_1)(Cx_1+D\xi_1)}e^{-i\pi(Ex_2+F\xi_2)(Gx_2+H\xi_2)}\\
	&\qquad\qquad\times \pi(\cA z_1)\mu(\cA)f\otimes \pi(\cB z_2)\mu(\cB)g\\
	&=e^{-2\pi i\tau}\rho(\cA z_1,\tau)\mu(\cA)f\otimes\rho(\cB z_2,\tau)\mu(\cB)g\\
	&=e^{-2\pi i\tau}\mu(\cA)\rho(z_1,\tau)f\otimes\mu(\cB)\rho(z_2,\tau)g\\
	&=e^{-2\pi i\tau}T(\rho(z_1,\tau)f\otimes \rho(z_2,\tau)g)\\
	&=T\rho(z,\tau)(f\otimes g)
\end{align*}
and (\ref{meta}) follows for tensor products. Next, if $F = \sum_{j=1}^\infty c_jf_j\otimes g_j$, $(c_j)_j\subseteq\bC$ definitely zero, $f_j,g_j\in L^2(\rd)$ ($j=1,2,\ldots$),
\begin{align*}
	\rho(\cC z,\tau)TF&=\rho(\cC z,\tau)T(\sum_jc_jf_j\otimes g_j)=\rho(\cC z,\tau)\sum_jc_jT(f_j\otimes g_j)\\
	&=\sum_jc_j\rho(\cC z,\tau)T(f_j\otimes g_j)=\sum_jc_jT\rho(z,\tau)(f_j\otimes g_j)=T\rho(z,\tau)F,
\end{align*}
and the assertion follows by a standard density argument.
	\end{proof}

	\begin{remark}
		Under the same notation as in Theorem \ref{thmOtimes}, if $\cA=I_{d\times d}$, then
		\begin{equation}\label{ce1}
			f\otimes\mu(\cB) g = \mu(\cC_1)(f\otimes g),
		\end{equation}
		where
		\[
			\cC_1=\begin{pmatrix}
				I_{d\times d} & 0_{d\times d} & 0_{d\times d} & 0_{d\times d}\\
				0_{d\times d} & E & 0_{d\times d} & F\\
				0_{d\times d} & 0_{d\times d} & I_{d\times d} & 0_{d\times d}\\
				0_{d\times d} & G & 0_{d\times d} & H
			\end{pmatrix}.
		\]
	If $\cB=I_{d\times d}$, we infer
		\[
			\mu(\cA) f\otimes g = \mu(\cC_2)(f\otimes g),
		\]
		where
		\[
	\cC_2=\begin{pmatrix}
			A & 0_{d\times d} & B & 0_{d\times d}\\
			0_{d\times d} & I_{d\times d} & 0_{d\times d} & 0_{d\times d}\\
			C & 0_{d\times d} & D & 0_{d\times d}\\
			0_{d\times d} & 0_{d\times d} & 0_{d\times d} & I_{d\times d}
		\end{pmatrix}.
	\]
	Observe that $\cC=\cC_1\cC_2=\cC_2\cC_1$.
	\end{remark}

\end{appendix}


\begin{thebibliography}{10}
	
	
%
%
\bibitem{Filippo2017}
\newblock G. Alberti, S. Dahlke, F.  De Mari, E.  De Vito,  and S.  Vigogna. 
\newblock Continuous and
discrete frames generated by the evolution flow of the Schr\"{o}dinger equation. \emph{Analysis and Appications}, 15(6): 915–937, 2017.
 \bibitem{BCGT2022} D. Bayer, E. Cordero, K. Gr\"ochenig and S. I. Trapasso. Linear perturbations of the Wigner transform and the Weyl quantization.In {\em Advances in microlocal and time-frequency analysis}, Appl. Numer. Harmon. Anal., pages 79--120. Birkh\"{a}user/Springer, Cham, 2020.
 \bibitem{BGGO2005} A. B\'enyi, L. Grafakos, K. Gr\"ochenig and  K.A.Okoudjou.
 \newblock{A class of Fourier multipliers for modulation spaces}.
 {\em Appl.  Comput. Harmon. Anal.}, 19(1):131--139, 2005.
\bibitem{KB2020} A. B\'enyi and K.A.Okoudjou. {\it Modulation Spaces With Applications to Pseudodifferential Operators and Nonlinear Schr\"{o}dinger Equations}, Springer New York, 2020.

\bibitem{bhimanikasso2020}
D. G. Bhimani, and M. Grillakis and K. A. Okoudjou.
\newblock{The Hartree–Fock equations in modulation spaces}.
{\em Communications in Partial Differential Equations}, 45(9):1088-1117, 2020.
\bibitem{BHIMANI2021107995}
	D. G. Bhimani, R. Manna, F. Nicola, S. Thangavelu and S. I. Trapasso.
	\newblock	Phase space analysis of the Hermite semigroup and applications to nonlinear global well-posedness,
	{\it Advances in Mathematics}, 392, 107995,	2021.
	
	
	
	
	
	
	
	

	\bibitem{Cohen1} L. Cohen, Generalized phase-space distribution functions,
	J. Math. Phys., 7:781--786, 1966.
	\bibitem{wiener8} F. Concetti, G. Garello and J. Toft. Trace ideals for Fourier integral operators with non-smooth symbols II. {\em Osaka J. Math.}, 47(3):739--786, 2010.
		\bibitem{CGR2022}
	E.~Cordero, G. Giacchi and L. Rodino.
	Wigner Analysis of Operators. Part II: Schr\"{o}dinger equations. {\it Submitted}. \texttt{arXiv:2208.00505}.
	\bibitem{CGNRJMPA} E.~Cordero, K.~Gr\"ochenig, F. Nicola and L. Rodino. Wiener algebras of Fourier integral operators. {\em J. Math. Pures Appl. (9)}, 99(2):219--233, 2013
\bibitem{CNR2015} E.~Cordero,  F. Nicola and  L. Rodino. Integral Representations for the Class of Generalized Metaplectic Operators.  {\em J. Fourier Anal. Appl.}, 21(4):694--714, 2015
\bibitem{FIO1}
E.~Cordero, F. Nicola and L. Rodino. Time-frequency
analysis of Fourier integral operators. {\it Commun. Pure
	Appl. Anal}., 9(1):1--21, 2010.
	\bibitem{Elena-book} E. Cordero and L. Rodino, \emph{ Time-Frequency Analysis of Operators}. De Gruyter Studies in Mathematics, 2020.
	\bibitem{CR2021}
	E.~Cordero and N. Rodino. {\it Wigner Analysis of Operators. Part I: Pseudodifferential Operators and Wave Front Sets}. Appl. Comput. Harmon. Anal. 58:85--123,  2022.

	
	

\bibitem{CR2022}
E. Cordero, L. Rodino. \emph{Characterization of modulation spaces by symplectic representations and applications to Schrödinger equations}. \emph{Submitted}. \texttt{arXiv:2204.14124}.
\bibitem{CT} E. Cordero and S. I. Trapasso. Linear perturbations of the Wigner distribution and the Cohen's class. {\em Anal. Appl. (Singap.)}, 18(3):385--422, 2020.
\bibitem{Filippocoorbite}
\newblock S. Dahlke, F. De Mari, E. De Vito, D. Labate, G. Steidl, G. Teschke,and S. Vigogna.
\newblock Coorbit spaces with voice in a Fr\'{e}chet space. \emph{J. Fourier Anal. Appl.},  23(1):141--206, 2016.
\bibitem{Dopico2009} F. M. Dopico,  and C. R. Johnson,
\newblock {Parametrization of the Matrix Symplectic Group and Applications},
\newblock {\em SIAM Journal on Matrix Analysis and Applications}, 31(2):650-673, 2009.
	\bibitem{ES89} J.J. Eggermont and G.M. Smith. Characterizing auditory neurons using the Wigner and Rihacek distributions: a comparison. {\em J. Acoust. Soc. Am.},  87(1):246--259, 1990.
	\bibitem{F1}  H.~G.~Feichtinger,
	\newblock Modulation spaces on locally
	compact abelian groups,
	\newblock {\em Technical Report, University Vienna, 1983,} and also in
	\newblock {\em Wavelets and Their Applications},
	M. Krishna, R. Radha,  S. Thangavelu, editors,
	\newblock Allied Publishers,  99--140,  2003. 
	\bibitem{Feichtinger_1981_Banach}
	H.~G. Feichtinger.
	\newblock Banach spaces of distributions of {W}iener's type and interpolation.
	\newblock In {\em Functional analysis and approximation ({O}berwolfach, 1980)},
	volume~60 of {\em Internat. Ser. Numer. Math.}, pages 153--165. Birkh\"auser,
	Basel-Boston, Mass., 1981.
		\bibitem{feichtinger-wiener-type}
	H.~G. Feichtinger.
	\newblock Banach convolution algebras of {W}iener type.
	\newblock In {\em Functions, series, operators, Vol. I, II (Budapest, 1980)},
	pages 509--524. North-Holland, Amsterdam, 1983.
	
	\bibitem{Feichtinger_1990_Generalized}
	H.~G. Feichtinger.
	\newblock Generalized amalgams, with applications to {F}ourier transform.
	\newblock {\em Canad. J. Math.}, 42(3):395--409, 1990.
%
	\bibitem{folland89}
	G.~B. Folland.
	\newblock {\em Harmonic analysis in phase space}.
	\newblock Princeton Univ. Press, Princeton, NJ, 1989.

\bibitem{Fuhr}
H. F\"uhr, I. Shafkulovska (2022). \emph{The metaplectic action on modulation spaces}.  Preprint \texttt{arXiv:2211.08389}.

%
	\bibitem{Galperin2004}
	Y.~V. Galperin and S.~Samarah.
	\newblock Time-frequency analysis on modulation spaces {$M^{p,q}_m$}, {$0<p,\
		q\leq\infty$}.
	\newblock {\em Appl. Comput. Harmon. Anal.}, 16(1):1--18, 2004.
	
	\bibitem{Giacchi}
	G. Giacchi. \emph{Metaplectic Wigner Distributions}. \emph{Submitted} \texttt{arXiv:2212.06818v2}.
	
	\bibitem{Gos11}
M.~A. de~Gosson.
\newblock {\em Symplectic methods in harmonic analysis and in mathematical
	physics}, volume~7 of {\em Pseudo-Differential Operators. Theory and
	Applications}.
\newblock Birkh\"auser/Springer Basel AG, Basel, 2011.
%

%
	
	\bibitem{book}
	K.~Gr{\"o}chenig. {\it Foundations of time-frequency
		analysis}. Birkh\"auser Boston, Inc., Boston, MA, 2001.
	
	
%
	
%
\bibitem{Guo-Zhao22} W. Guo and G. Zhao. Characterization of Boundedness on Wiener Amalgam
Spaces of Multilinear Rihaczek Distributions. Preprint. \texttt{arXiv:2201.06434}.
%
%
	

\bibitem{Kadison}
R. V. Kadison, J. R. Ringrose. \emph{Fundamentals of the theory of operator algebras. Volume I: Elementary theory.} New York: Academic press, 1983.

\bibitem{Kobayashi2006}
M.~Kobayashi.
\newblock Modulation spaces {$M^{p,q}$} for {$0<p,q\leq\infty$}.
\newblock {\em J. Funct. Spaces Appl.}, 4(3):329--341, 2006.
\bibitem{LUEF2018288}
F. Luef and  E. Skrettingland.
\newblock{Convolutions for localization operators}.
\newblock{\em Journal de Mathématiques Pures et Appliquées}, 118:288--316, 2018.
%
\bibitem{MRomero2022}
F. Marceca and J. L. Romero.
\newblock
Spectral deviation of concentration operators for the Short-time Fourier Transform.
 Preprint. \texttt{arXiv:2104.06150}.
\bibitem{MO2002}
H.~Morsche and P.J. Oonincx.
\newblock On the Integral Representations for Metaplectic Operators.
\newblock {\it J. Fourier Anal. Appl.}, 8(3):245--257, 2002.
	\bibitem{PILIPOVIC2004194}
S. Pilipović and N. Teofanov.
\newblock Pseudodifferential operators on ultra-modulation spaces
\newblock {\em Journal of Functional Analysis}, 208(1):194--228, 2004.
	\bibitem{Rauhut2007Winer}
H.~Rauhut.
\newblock Wiener amalgam spaces with respect to quasi-{B}anach spaces.
\newblock {\em Colloq. Math.}, 109(2):345--362, 2007.

	\bibitem{RudinFunc}
	W. Rudin, Functional Analysis (second edition), McGraw Hill Education, 1991. 
	
	\bibitem{RSTT2011}
	M. Ruzhansky,  M. Sugimoto,  J. Toft and  N. Tomita,  Changes of variables in modulation and Wiener amalgam spaces. \emph{Mathematische Nachrichten}, 284(16):2078--2092, 2011.
	\bibitem{Shafi}
	I. Shafi, J. Ahmad, S. I. Shah, F. M. Kashif. \emph{Techniques to obtain good resolution and concentrated time-frequency distributions: a review.} EURASIP Journal on Advances in Signal processing, 2009(1), 673539, 2009.
	\bibitem{sugimototomita}
	M. Sugimoto and N. Tomita.
	\newblock {{T}he  dilation property of modulation spaces and their inclusion relation with Besov spaces}.
	\newblock {\em J. Funct. Anal.},  248(1):79–106, 2007.
	

%
%


%
%
	
	
	
	
	
	
	\bibitem{Nenad2018}
	N.~Teofanov.
	\newblock Bilinear localization operators on modulation spaces.
	\newblock {\em J. Funct. Spaces}, pages Art. ID 7560870, 10, 2018.
	\bibitem{toft1}
J.~Toft.
\newblock Continuity properties for modulation spaces, with applications to
pseudo-differential calculus. {I}.
\newblock {\em J. Funct. Anal.}, 207(2):399--429, 2004.
	\bibitem{wang} B. Wang, Z. Huo, C. Hao and Z. Guo. \textit{Harmonic Analysis Method for Nonlinear Evolution Equations. I.} World Scientific Publishing Co. Pte. Ltd., Hackensack, NJ, 2011. 
	
	
%
	
%
%
	\bibitem{ZJQ21}
Z. C. Zhang, X. Jiang, S. Z. Qiang, A. Sun, Z. Y. Liang, X. Y. Shi, and
A. Y. Wu. { Scaled Wigner distribution using fractional instantaneous
autocorrelation}. {\it Optik}, 237, 166691,  2021.
%
\bibitem{Zhang21bis} Z. C. Zhang. { Uncertainty principle of complex-valued functions in specific free metaplectic transformation domains}. {\it J. Fourier Anal. Appl.}, 27(4):68, 2021.
\bibitem{Zhang21} Z. C. Zhang, X. Y. Shi, A. Y. Wu, and D. Li. { Sharper N-D Heisenberg's uncertainty principle}. {\it IEEE 	Signal Process. Lett.}, 28(7):1665--1669,  2021.
\end{thebibliography}
\end{document}